\documentclass[smallextended]{svjour3}       
\smartqed  
\usepackage{graphicx}
\usepackage{amsmath}
\usepackage{amssymb}
\usepackage{amsfonts}
\usepackage{graphicx}
\usepackage{epstopdf}
\usepackage{url}
\usepackage{hyperref}
\usepackage{booktabs}
\usepackage{algorithm}
\usepackage{algpseudocode}
\usepackage{multirow}
\usepackage{enumitem}
\newtheorem{assumption}{Assumption} 

\newcommand{\R}{\mathbb{R}}
\newcommand{\Rext}{\mathbb{R}\cup\{+\infty\}}

\newcommand{\set}[1]{\left\{#1\right\}}
\newcommand{\sets}[1]{\{#1\}}
\newcommand{\norm}[1]{\left\Vert #1\right\Vert}
\newcommand{\norms}[1]{\Vert #1\Vert}
\newcommand{\argmin}{\mathrm{arg}\!\min}

\newcommand{\Eproof}{\hfill$\square$}

\newcommand{\Gc}{\mathcal{G}}

\newcommand{\Xc}{\mathcal{X}}
\newcommand{\Yc}{\mathcal{Y}}

\newcommand{\Rc}{\mathcal{R}}

\newcommand{\Lc}{\mathcal{L}}
\newcommand{\Fc}{\mathcal{F}}
\newcommand{\Tc}{\mathcal{T}}
\newcommand{\Vc}{\mathcal{V}}
\newcommand{\dom}[1]{\mathrm{dom}\left(#1\right)}

\newcommand{\iprod}[1]{\left\langle #1\right\rangle}
\newcommand{\iprods}[1]{\langle #1\rangle}

\newcommand{\prox}{\textrm{prox}}

\newcommand{\BigO}[1]{\mathcal{O}\left(#1\right)}

\newcommand{\beforesubsec}{\vspace{-3.5ex}}
\newcommand{\aftersubsec}{\vspace{-2ex}}
\newcommand{\beforesec}{\vspace{-3ex}}
\newcommand{\aftersec}{\vspace{-2ex}}

\begin{document}

\title{
A Unified Convergence Rate Analysis of The Accelerated Smoothed Gap Reduction Algorithm
}

\titlerunning{
A Unified Convergence Rate Analysis of The ASGARD Algorithm
}        

\author{Quoc Tran-Dinh$^{*}$}

\authorrunning{Q. Tran-Dinh} 

\institute{Quoc Tran-Dinh \at
              Department of Statistics and Operations Research\\ 
              The University of North Carolina at Chapel Hill\\
              333 Hanes Hall, UNC-CH, Chapel Hill, NC27599. \\
              \email{quoctd@email.unc.edu}  
}

\date{Received: date / Accepted: date}

\maketitle

\begin{abstract}
In this paper, we develop a unified convergence analysis framework for the \textit{Accelerated Smoothed GAp ReDuction algorithm} (ASGARD) introduced in \cite[\textit{Tran-Dinh et al}]{TranDinh2015b}.
Unlike \cite{TranDinh2015b}, the new analysis covers three settings in a single algorithm: general convexity, strong convexity, and strong convexity and smoothness.
Moreover, we establish the convergence guarantees on three criteria: (i) gap function, (ii) primal objective residual, and (iii) dual objective residual.
Our convergence rates are optimal (up to a constant factor) in all cases.
While the convergence rate on the primal objective residual for the general convex case has been established in \cite{TranDinh2015b}, 
we prove additional convergence rates on the gap function and the dual objective residual.
The analysis for the last two cases is completely new.
Our results provide a complete picture on the convergence guarantees of ASGARD.
Finally, we present four different numerical experiments on a representative optimization model to verify our algorithm and compare it with Nesterov's smoothing technique.

\end{abstract}

\beforesec
\section{Introduction}\label{sec:intro}
\aftersec
We consider the following classical convex-concave saddle-point problem:
\begin{equation}\label{eq:saddle_point_prob}
\min_{x\in\R^p}\max_{y\in\R^n}\Big\{ \Lc(x, y) := f(x) + \iprods{Kx, y} - g^{*}(y) \Big\},
\end{equation}
where $f : \R^p \to \Rext$ and $g : \R^n\to\Rext$ are proper, closed, and convex, $K : \R^p \to  \R^n$  is a linear operator, and $g^{*}(y) := \sup_{x}\set{ \iprods{y, x} - g(x)}$ is the Fenchel conjugate of $g$.

The convex-concave saddle-point problem \eqref{eq:saddle_point_prob} can be written in the primal and dual forms.
The primal problem is defined as
\begin{equation}\label{eq:primal_prob}
F^{\star} := \min_{x\in\R^p}\Big\{F(x) :=  f(x) + g(Kx)  \Big\}.
\end{equation}
The corresponding dual problem is written as
\begin{equation}\label{eq:dual_prob}
D^{\star} := \min_{y\in\R^n}\Big\{ D(y) := f^{*}(-K^{\top}y) + g^{*}(y) \Big\}.
\end{equation}
Clearly,  both the primal problem \eqref{eq:primal_prob} and its dual form  \eqref{eq:dual_prob} are convex.

\vspace{0.75ex}
\noindent\textit{\textbf{Motivation.}}
In \cite{TranDinh2015b}, two accelerated smoothed gap reduction algorithms are proposed to solve \eqref{eq:primal_prob} and its special constrained convex problem.
Both algorithms achieve optimal sublinear convergence rates (up to a constant factor) in the sense of black-box first-order oracles \cite{Nemirovskii1983,Nesterov2004}, when $f$ and $g$ are convex, and when $f$ is strongly convex and $g$ is just convex, respectively.
The first algorithm in  \cite{TranDinh2015b} is called ASGARD (\textit{Accelerated Smoothed GAp ReDuction}).
To the best of our knowledge, except for a special case \cite{TranDinh2012a}, ASGARD was  the first primal-dual first-order algorithm that achieves a non-asymptotic optimal convergence rate on the last primal iterate.
ASGARD is also different from the alternating direction method of multipliers (ADMM) and its variants, where it does not require solving complex subproblems but uses the proximal operators of $f$ and $g^{*}$.
However, ASGARD (i.e. \cite[Algorithm 1]{TranDinh2015b}) only covers the general convex case, and it needs only one proximal operation of $f$ and of $g^{*}$ per iteration.
To handle the strong convexity of $f$, a different variant is developed in \cite{TranDinh2015b}, called ADSGARD, but requires two proximal operations of $f$ per iteration.
Therefore, the following natural question is arising: 
\begin{itemize}
\vspace{-0.5ex}
\item[] ``\textit{Can we develop a unified variant of ASGARD that covers three settings: general convexity, strong convexity, and strong convexity and smoothness?}''
\end{itemize}
\noindent\textit{\textbf{Contribution.}}
In this paper, we affirmatively answer this question by developing a unified variant of ASGARD that covers the following three settings: 
\begin{itemize}
\item[] Case 1: Both $f$ and $g^{*}$ in \eqref{eq:saddle_point_prob} are only convex, but not strongly convex.
\item[] Case 2: Either $f$ or $g^{*}$ is strongly convex, but not both $f$ and $g^{*}$.
\item[] Case 3: Both $f$ and $g^{*}$ are strongly convex.
\end{itemize}
The new variant only requires one proximal operation of $f$ and of $g^{*}$ at each iteration as in the original ASGARD and existing primal-dual methods, e.g., in \cite{boct2020variable,Chambolle2011,Chen2013a,Condat2013,Esser2010,connor2014primal,Goldstein2013,vu2013variable}.
Our algorithm reduces to ASGARD in Case 1, but uses a different update rule for $\eta_k$ (see Step~\ref{step:i2} of Algorithm~\ref{alg:A1}) compared to ASGARD. 
In Case 2 and Case 3, our algorithm is completely new by incorporating the strong convexity parameter $\mu_f$ of $f$ and/or $\mu_{g^{*}}$ of $g^{*}$ in the parameter update rules to achieve optimal $\BigO{1/k^2}$ sublinear and $(1 - \BigO{1/\sqrt{\kappa_F}})$ linear rates, respectively, where $k$ is the iteration counter and $\kappa_F := \norms{K}^2/(\mu_f\mu_{g^{*}})$.

In terms of convergence guarantees, we establish  that, in all cases, our algorithm achieves optimal  rates on the last primal sequence $\sets{x^k}$ and an averaging dual sequence $\set{\tilde{y}^k}$.
Moreover, our convergence guarantees are on three different criteria: (i) the gap function for \eqref{eq:saddle_point_prob}, (ii) the primal objective residual $F(x^k) - F^{\star}$ for \eqref{eq:primal_prob}, and (iii) the dual objective residual $D(\tilde{y}^k) - D^{\star}$ for \eqref{eq:dual_prob}.
Our paper therefore provides a unified and full analysis on convergence rates of ASGARD for solving three problems \eqref{eq:saddle_point_prob}, \eqref{eq:primal_prob}, and \eqref{eq:dual_prob} simultaneously.

We emphasize that primal-dual first-order methods for solving \eqref{eq:saddle_point_prob}, \eqref{eq:primal_prob}, and \eqref{eq:dual_prob}, and their convergence analysis have been well studied in the literature. 
To avoid overloading this paper, we refer to our recent works \cite{TranDinh2015b,tran2019non} for a more thorough discussion and comparison between existing methods.
Hitherto, there have been numerous papers studying convergence rates of primal-dual first-order methods, including \cite{boct2020variable,Chambolle2011,chambolle2016ergodic,Chen2013a,Davis2014a,He2012,valkonen2019inertial}.
However, the best known and optimal rates are only achieved via averaging or weighted averaging sequences, which are also known as ergodic rates.
The convergence rates on the last iterate sequence are often slower and suboptimal.
Recently, the optimal convergence rates of the last iterates have been studied for primal-dual first-order methods, including \cite{TranDinh2015b,tran2019non,valkonen2019inertial}.
As pointed out in \cite{Chambolle2011,Davis2015,sabach2020faster,TranDinh2015b,tran2019non}, the last iterate convergence guarantee is very important in various applications to maintain some desirable structures of the final solutions such as sparsity, low-rankness, or sharp-edgedness in images.  
This also motivates us to develop ASGARD.

\vspace{0.75ex}
\noindent\textit{\textbf{Paper outline.}}
The rest of this paper is organized as follows.
Section~\ref{sec:preliminaries} recalls some basic concepts, states our assumptions, and characterizes the optimality condition of \eqref{eq:saddle_point_prob}. 
Section~\ref{sec:alg_scheme} presents our main results on the algorithm and its convergence analysis.
Section~\ref{sec:num_exp} provides a set of experiments to verify our theoretical results and compare our method with Nesterov's smoothing scheme in \cite{Nesterov2005c}.
Some technical proofs are deferred to the appendix.  

\beforesec
\section{Basic Concepts, Assumptions, and Optimality Condition}\label{sec:preliminaries}
\aftersec
We are working with Euclidean spaces, $\R^p$ and $\R^n$, equipped with the standard inner product $\iprods{\cdot,\cdot}$ and the Euclidean norm $\norms{\cdot}$.
We will use the Euclidean norm for the entire paper.
Given a proper, closed, and convex function $f$, we use $\dom{f}$ and $\partial{f}(x)$ to denote its domain and its subdifferential at $x$, respectively. 
We also use $\nabla{f}(x)$ for a subgradient or the gradient (if $f$ is differentiable) of $f$ at $x$.
We denote by $f^{*}(y) := \sup\set{\iprod{y, x} - f(x) : x\in\dom{f}}$ the Fenchel conjugate of $f$.
We denote by $ \mathrm{ri}(\Xc)$ the relative interior of $\Xc$. 

A function $f$ is called $M_f$-Lipschitz continuous if $\vert f(x) - f(\tilde{x})\vert \leq M_f\norms{x - \tilde{x}}$ for all $x, \tilde{x} \in \dom{f}$, where $M_f \in [0, +\infty)$ is called a Lipschitz constant.
A proper, closed, and convex function $f$ is $M_f$-Lipschitz continuous if and only if $\partial{f}(\cdot)$ is uniformly bounded by $M_f$ on $\dom{f}$.
For a smooth function $f$, we say that $f$ is $L_f$-smooth (or Lipschitz gradient) if for any $x, \tilde{x}\in\dom{f}$, we have $\Vert\nabla{f}(x) - \nabla{f}(\tilde{x})\Vert \leq L_f\norms{x - \tilde{x}}$, where $L_f \in [0, +\infty)$ is a Lipschitz constant. 
A function $f$ is called $\mu_f$-strongly convex with a strong convexity parameter $\mu_f \geq 0$ if $f(\cdot) - \frac{\mu_f}{2}\norms{\cdot}^2$ remains convex.
For a proper, closed, and convex function $f$, $\prox_{\gamma f}(x) := \argmin\big\{f(\tilde{x}) + \tfrac{1}{2\gamma}\norms{\tilde{x} - x}^2 : \tilde{x}\in\dom{f} \big\}$ is called the proximal operator of $\gamma f$, where $\gamma > 0$.

\beforesubsec
\subsection{\textbf{Basic assumptions and optimality condition}}\label{subsec:primal_dual_formulation}
\aftersubsec
In order to show the relationship between \eqref{eq:saddle_point_prob}, \eqref{eq:primal_prob} and its dual form \eqref{eq:dual_prob}, we require the following assumptions.

\begin{assumption}\label{as:A1}
The following assumptions hold for \eqref{eq:saddle_point_prob}.
\begin{itemize}
\item[$\mathrm{(a)}$] Both functions $f$ and $g$ are proper, closed, and convex on their domain.
\item[$\mathrm{(b)}$] There exists a saddle-point $z^{\star} := (x^{\star}, y^{\star})$ of $\Lc$ defined in \eqref{eq:saddle_point_prob}, i.e.:
\begin{equation}\label{eq:saddle_point}
\Lc(x^{\star}, y) \leq \Lc(x^{\star}, y^{\star}) \leq \Lc(x, y^{\star}), \quad \forall (x, y)\in\dom{f}\times\dom{g^{*}}.
\end{equation}
\item[$\mathrm{(c)}$] The Slater condition $0 \in \mathrm{ri}(\dom{g} - K\dom{f})$ holds.
\end{itemize}
\end{assumption}
Assumption~\ref{as:A1} is standard in convex-concave saddle-point settings.
Under Assumption~\ref{as:A1}, strong duality holds, i.e. $F^{\star} = \Lc(x^{\star}, y^{\star}) = -D^{\star}$.

To characterize a saddle-point of \eqref{eq:saddle_point_prob}, we define the following gap function:
\begin{equation}\label{eq:gap_func}
\Gc_{\Xc\times\Yc}(x, y) := \sup\big\{ \Lc(x, \tilde{y}) - \Lc(\tilde{x}, y) : \tilde{x}\in\Xc, \ \tilde{y}\in\Yc \big\},
\end{equation}
where $\Xc\subseteq\dom{f}$ and $\Yc\subseteq\dom{g^{*}}$ are nonempty, closed, and convex subsets such that $\Xc\times\Yc$ contains a saddle-point $(x^{\star},y^{\star})$ of \eqref{eq:saddle_point_prob}.
Clearly, we have $\Gc_{\Xc\times\Yc}(x, y) \geq 0$ for all $(x, y)\in\Xc\times\Yc$.
Moreover, if $(x^{\star}, y^{\star})$ is a saddle-point of \eqref{eq:saddle_point_prob} in $\Xc\times \Yc$, then $\Gc_{\Xc\times\Yc}(x^{\star}, y^{\star}) = 0$.

\beforesubsec
\subsection{\textbf{Smoothing technique for $g$}}\label{subsec:proximal_smoothing}
\aftersubsec
We first smooth $g$ in \eqref{eq:primal_prob} using Nesterov's smoothing technique \cite{Nesterov2005c} as
\begin{equation}\label{eq:smoothed_g}
g_{\beta}(u, \dot{y}) := \max_{y\in\R^n}\set{\iprods{u, y} - g^{*}(y) - \tfrac{\beta}{2}\norms{y  - \dot{y}}^2},
\end{equation}
where $g^{*}$ is the Fenchel conjugate of $g$, $\beta > 0$ is a smoothness parameter, and $\dot{y}$ is a given proximal center.
We denote by $\nabla_ug_{\beta}(u) = \prox_{g^{*}/\beta}(\dot{y} + \frac{1}{\beta}u)$ the gradient of $g_{\beta}$ w.r.t. $u$.

Given $g_{\beta}$ defined by \eqref{eq:smoothed_g}, we can approximate $F$ in \eqref{eq:primal_prob} by
\begin{equation}\label{eq:F_beta}
F_{\beta}(x, \dot{y}) := f(x) + g_{\beta}(Kx, \dot{y}).
\end{equation}
The following lemma provides two key inequalities to link $F_{\beta}$ to $\Lc$ and $D$.

\begin{lemma}\label{le:basic_properties}
Let $F_{\beta}$ be defined by \eqref{eq:F_beta} and $(x^{\star}, y^{\star})$ be a saddle-point of \eqref{eq:saddle_point_prob}.
Then, for any $x, \bar{x}\in\dom{f}$ and $\dot{y}, \tilde{y}, y \in\dom{g^{*}}$, we have
\begin{equation}\label{eq:key_est01}
\begin{array}{lcl}
\Lc(\bar{x}, y)  & \leq &  F_{\beta}(\bar{x}, \dot{y}) + \frac{\beta}{2}\norms{y - \dot{y}}^2, \vspace{1ex}\\
D(\tilde{y}) - D(y^{\star}) & \leq & F_{\beta}(\bar{x}, \dot{y}) - \Lc(\tilde{x}, \tilde{y}) + \frac{\beta}{2}\norms{y^{\star} - \dot{y}}^2, \quad \forall\tilde{x} \in \partial{f^{*}}(-K^{\top}\tilde{y}).
\end{array}
\end{equation}
\end{lemma}

\begin{proof}
Using the definition of $\Lc$ in \eqref{eq:saddle_point_prob} and of $F_{\beta}$ in \eqref{eq:F_beta}, we have $\Lc(\bar{x}, y) =  f(\bar{x}) +  \iprods{K\bar{x}, y} - g^{*}(y) \leq f(\bar{x}) + \sup_{y}\sets{\iprods{K\bar{x}, y} - g^{*}(y) - \frac{\beta}{2}\norms{y - \dot{y}}^2}  +  \frac{\beta}{2}\norms{y - \dot{y}}^2 = F_{\beta}(\bar{x}, \dot{y}) + \frac{\beta}{2}\norms{y - \dot{y}}^2$,
which proves the first line of \eqref{eq:key_est01}. 

Next, for any $\tilde{x} \in \partial{f^{*}}(-K^{\top}\tilde{y})$, we have $f^{*}(-K^{\top}\tilde{y}) = \iprods{-K^{\top}\tilde{y}, \tilde{x}} - f(\tilde{x})$ by Fenchel's equality \cite{Bauschke2011}.
Hence, $D(\tilde{y}) = f^{*}(-K^{\top}\tilde{y}) + g^{*}(\tilde{y}) = -\Lc(\tilde{x}, \tilde{y})$.
On the other hand, by \eqref{eq:saddle_point}, we have $\Lc(\bar{x}, y^{\star}) \geq \Lc(x^{\star}, y^{\star}) = -D(y^{\star})$.
Combining these two expressions and the first line of \eqref{eq:key_est01}, we obtain $D(\tilde{y}) - D(y^{\star}) \leq \Lc(\bar{x}, y^{\star}) - \Lc(\tilde{x},\tilde{y}) \leq F_{\beta}(\bar{x}, \dot{y}) - \Lc(\tilde{x}, \tilde{y}) + \frac{\beta}{2}\norms{\dot{y} - y^{\star}}^2$.
\Eproof
\end{proof}

\beforesec
\section{New ASGARD Variant and Its Convergence Guarantees}\label{sec:alg_scheme}
\aftersec
In this section, we derive a new and unified variant of  ASGARD in \cite{TranDinh2015b} and analyze its convergence rate guarantees for three settings.

\beforesubsec
\subsection{\textbf{The derivation of algorithm and one-iteration analysis}}
\aftersubsec
Given $\dot{y} \in\R^n$ and $\hat{x}^k \in\R^p$, the main step of ASGARD consists of one primal and one dual updates as follows:
\begin{equation}\label{eq:prox_grad_step0}
\left\{\begin{array}{lcl}
y^{k+1} & := &  \prox_{g^{*}/\beta_k}\big(\dot{y} + \tfrac{1}{\beta_k}K\hat{x}^k \big), \vspace{1ex}\\
x^{k+1}           &  := & \prox_{f/L_k}\big(\hat{x}^k - \tfrac{1}{L_k}K^{\top}y^{k+1} \big),  
\end{array}\right.
\end{equation}
where $\beta_k > 0$ is the smoothness parameter of $g$, and $L_k > 0$ is an estimate of the Lipschitz constant of $\nabla{g_{\beta_k}}$.
Here, \eqref{eq:prox_grad_step0} serves as basic steps of various primal-dual first-order methods in the literature, including \cite{Chambolle2011,He2012,TranDinh2015b}.
However, instead of updating $\dot{y}$, ASGARD fixes it for all iterations.

The following lemma serves as a key step for our analysis in the sequel.
Since its statement and proof are rather different from \cite[Lemma 2]{TranDinh2015b}, we provide its proof in Appendix~\ref{apdx:le:descent_pro}. 

\begin{lemma}[\cite{TranDinh2015b}]\label{le:descent_pro}
Let $(x^{k+1}, y^{k+1})$ be generated by \eqref{eq:prox_grad_step0},  $F_{\beta}$ be defined by \eqref{eq:F_beta}, and $\Lc$ be given by \eqref{eq:saddle_point_prob}.
Then, for any $x\in\dom{f}$ and $\tau_k \in [0,1]$, we have
\begin{equation}\label{eq:key_est00_a}
\hspace{-2ex}
\arraycolsep=0.1em
\begin{array}{lcl}
F_{\beta_k}(x^{k+1},\dot{y})  & \leq & (1 - \tau_k) F_{\beta_{k-1}}(x^k, \dot{y})  + \tau_k\Lc(x, y^{k+1}) \vspace{1ex}\\
&& + {~} \frac{L_k\tau_k^2}{2} \big\Vert \tfrac{1}{\tau_k}[\hat{x}^k - (1-\tau_k)x^k] - x \big\Vert^2 - \frac{\mu_f(1-\tau_k)\tau_k}{2}\norms{x - x^k}^2 \vspace{1ex}\\
&& - {~} \frac{\tau_k^2}{2}\left( L_k + \mu_f \right) \big\Vert \tfrac{1}{\tau_k}[x^{k+1} - (1-\tau_k)x^k] - x \big\Vert^2  \vspace{1ex}\\
&& - {~} \frac{L_k}{2}\norms{x^{k+1} - \hat{x}^k}^2 + \frac{1}{2(\mu_{g^{*}}+\beta_k)}\Vert K(x^{k+1} - \hat{x}^k)\Vert^2  \vspace{1ex}\\
&&  - {~} \frac{1-\tau_k}{2}\left[ \tau_k\beta_k - (\beta_{k-1} - \beta_k)\right]\norms{ \nabla_ug_{\beta_k}(Kx^k,\dot{y})  - \dot{y}}^2.
\end{array}
\hspace{-8ex}
\end{equation}
\end{lemma}

\noindent
Together with the primal-dual step \eqref{eq:prox_grad_step0}, we also apply Nesterov's accelerated step to $\hat{x}^k$ and an averaging step to $\tilde{y}^k$ as follows:
\begin{equation}\label{eq:prox_grad_scheme}
\left\{\begin{array}{lcl}
\hat{x}^{k+1} & := & x^{k+1} + \eta_{k+1}(x^{k+1} - x^k), \vspace{1ex}\\
\tilde{y}^{k+1} &:= & (1-\tau_k)\tilde{y}^k + \tau_ky^{k+1}, 
\end{array}\right.
\end{equation}
where $\tau_k \in (0, 1)$ and  $\eta_{k+1} > 0$ will be determined in the sequel. 

To analyze  the convergence of the new ASGARD variant, we define the following Lyapunov function (also called a potential function):
\begin{equation}\label{eq:lyapunov_func1}
\begin{array}{lcl}
\Vc_k(x)   & := &   F_{\beta_{k-1}}(x^k, \dot{y}) - \Lc(x, \tilde{y}^k)  \vspace{1ex}\\
&& + {~} \tfrac{(L_{k-1} + \mu_f)\tau_{k-1}^2}{2}\norms{\tfrac{1}{\tau_{k-1}}[x^k - (1 - \tau_{k-1})x^{k-1}] - x}^2.
\end{array}
\end{equation}
The following lemma provides a key recursive estimate to analyze the convergence of \eqref{eq:prox_grad_step0} and \eqref{eq:prox_grad_scheme}, whose proof is given in Appendix~\ref{apdx:le:maintain_gap_reduction1}.

\begin{lemma}\label{le:maintain_gap_reduction1}
Let $(x^{k}, \hat{x}^k, y^k, \tilde{y}^{k})$ be updated by \eqref{eq:prox_grad_step0} and \eqref{eq:prox_grad_scheme}.
Given $\beta_0 > 0$, $\tau_k, \tau_{k+1}\in (0, 1]$, let $\beta_k$, $L_k$, and $\eta_{k+1}$ be updated by
\vspace{-0.5ex}
\begin{equation}\label{eq:Lk_m_k}
\beta_k := \frac{\beta_{k-1}}{1+\tau_k}, \quad L_k := \frac{\norms{K}^2}{\beta_k + \mu_{g^{*}}}, \quad\text{and}\quad  \eta_{k+1} := \frac{(1-\tau_{k})\tau_{k}}{\tau_{k}^2 + m_{k+1}\tau_{k+1}},
\vspace{-0.5ex}
\end{equation}
where $m_{k+1} := \frac{L_{k+1} + \mu_f}{L_{k} + \mu_f}$.
Suppose further that $\tau_k \in (0, 1]$ satisfies
\begin{equation}\label{eq:param_cond}
\hspace{-1ex}
\arraycolsep=0.2em
\left\{\begin{array}{llcl}
&(L_{k-1} + \mu_f)(1-\tau_k)\tau_{k-1}^2 + \mu_f(1-\tau_k)\tau_k   &\geq &  L_k\tau_k^2, \vspace{1ex}\\
&(L_{k-1} + \mu_f)(L_k + \mu_f)\tau_k\tau_{k-1}^2 + (L_k + \mu_f)^2\tau_k^2 & \geq & (L_{k-1} + \mu_f)L_k\tau_{k-1}^2.
\end{array}\right.
\hspace{-1ex}
\end{equation}
Then, for any $x\in\dom{f}$, the Lyapunov function $\Vc_k$ defined by \eqref{eq:lyapunov_func1} satisfies
\begin{equation}\label{eq:key_est1_ncvx}
\Vc_{k+1}(x) \leq (1-\tau_k)\Vc_{k}(x).
\end{equation}
\end{lemma}

\vspace{1ex}
\noindent\textbf{The unified ASGARD algorithm.}
Our next step is to expand \eqref{eq:prox_grad_step0}, \eqref{eq:prox_grad_scheme}, and \eqref{eq:Lk_m_k} algorithmically to obtain a new ASGARD variant (called ASGARD+) as presented in Algorithm~\ref{alg:A1}.

\begin{algorithm}[ht!]\caption{(New Accelerated Smoothed GAp ReDuction (ASGARD+))}\label{alg:A1}
\normalsize
\begin{algorithmic}[1]
   \State{\bfseries Initialization:} \label{step:i0a} 
   Choose $x^0 \in\dom{f}$, $\tilde{y}^0\in\dom{g^{*}}$, and $\dot{y}\in \R^n$. 
   \vspace{0.5ex}
   \State\label{step:i0b} 
   Choose $\tau_0 \in (0, 1]$ and $\beta_0 > 0$.  
   Set $L_0 := \frac{\norms{K}^2}{\mu_{g^{*}} + \beta_0}$ and $\hat{x}^0 := x^0$.
   \vspace{0.5ex}
   \State{\bfseries For $k := 0, 1, \cdots, k_{\max}$ do}
   \vspace{0.5ex}   
   \State\hspace{2ex}\label{step:i1} 
    Update $\tau_{k+1}$ as in Theorem~\ref{th:convergence1}, \ref{th:convergence_1a}, or \ref{th:convergence_1b}, and update $\beta_{k+1} := \frac{\beta_{k}}{1+\tau_{k+1}}$.
   \vspace{0.5ex}   
    \State\hspace{2ex}\label{step:i2} 
    Let $L_{k+1} := \frac{\norms{K}^2}{\mu_{g^{*}} + \beta_{k+1}}$, $m_{k+1} := \frac{L_{k+1}+\mu_f}{L_{k} + \mu_f}$, and $\eta_{k+1} :=  \frac{(1-\tau_{k})\tau_{k}}{\tau_{k}^2 + m_{k+1}\tau_{k+1}}$.
    \vspace{0.5ex}   
     \State\hspace{2ex}\label{step:i3} Update 
     \begin{equation*}
     \left\{\begin{array}{lcl}
     	y^{k+1} & := &  \prox_{g^{*}/\beta_k}\big(\dot{y} + \tfrac{1}{\beta_k}K\hat{x}^k \big), \vspace{1ex}\\
	x^{k+1}           &  := & \prox_{f/L_k}\big(\hat{x}^k - \tfrac{1}{L_k}K^{\top}y^{k+1} \big), \vspace{1ex}\\
	\hat{x}^{k+1} & := & x^{k+1} + \eta_{k+1}(x^{k+1} - x^k), \vspace{1ex}\\
        \tilde{y}^{k+1} & := & (1-\tau_k)\tilde{y}^k + \tau_ky^{k+1}.
     \end{array}\right.
     \end{equation*}
   \State{\bfseries EndFor}
\end{algorithmic}
\end{algorithm}

Compared to the original ASGARD in \cite{TranDinh2015b}, Algorithm~\ref{alg:A1} requires one additional averaging dual step on $\tilde{y}^k$ at Step~\ref{step:i3} to obtain the dual convergence.
Note that Algorithm~\ref{alg:A1} also incorporates the strong convexity parameters $\mu_f$ of $f$ and $\mu_{g^{*}}$ of $g^{*}$ to cover three settings:  general convexity ($\mu_f = \mu_{g^{*}} = 0$), strong convexity ($\mu_f > 0$ and $\mu_{g^{*}} = 0$), and strong convexity and smoothness ($\mu_f > 0$ and $\mu_{g^{*}} > 0$).
More precisely, $L_k$ and the momentum step-size $\eta_{k+1}$ are also different from \cite{TranDinh2015b} by incorporating $\mu_{g^{*}}$ and $\mu_f$.
The per-iteration complexity of ASGSARD+ remains the same as ASGARD except for the averaging dual update $\tilde{y}^k$.
However, this step is not required if we only solve \eqref{eq:primal_prob}.
We highlight that  if we apply a new approach from \cite{tran2019non} to \eqref{eq:saddle_point_prob}, then we can also update the proximal center $\dot{y}$ at each iteration.

\beforesubsec
\subsection{\textbf{Case 1: Both $f$ and $g^{*}$ are just convex ($\mu_f = \mu_{g^{*}} = 0$)}}\label{subsec:general_cvx_ccv}
\aftersubsec
The following theorem establishes convergence rates of Algorithm~\ref{alg:A1} for the  general convex case where both $f$ and $g^{*}$ are just convex.

\begin{theorem}\label{th:convergence1}
Suppose that Assumption~\ref{as:A1} holds and both $f$ and $g^{*}$ are only convex, i.e. $\mu_f = \mu_{g^{*}} = 0$.
Let $\sets{(x^k, \tilde{y}^k)}$ be generated by Algorithm~\ref{alg:A1} for solving \eqref{eq:saddle_point_prob}, where $\tau_0 := 1$ and $\tau_{k+1}$ is the unique solution of the cubic equation $\tau^3 + \tau^2 + \tau_k^2\tau - \tau_k^2 = 0$ in $\tau$, which always exists.
Then, for all $k \geq 1$, we have:
\begin{itemize}
\item[$\mathrm{(a)}$]
The gap function $\Gc_{\Xc\times\Yc}$ defined by \eqref{eq:gap_func} satisfies
\begin{equation}\label{eq:main_result1}
\Gc_{\Xc\times\Yc}(x^k, \tilde{y}^k) \leq \frac{\norms{K}^2}{2\beta_0k} \sup_{x\in\Xc}\norms{x^0 - x}^2 + \frac{\beta_0}{k+1}\sup_{y\in\Yc}\norms{y - \dot{y}}^2.
\end{equation}
\item[$\mathrm{(b)}$]
If $g$ is $M_g$-Lipschitz continuous on $\dom{g}$, then for \eqref{eq:primal_prob}, it holds that
\begin{equation}\label{eq:main_result1b}
F(x^k) - F^{\star} \leq \frac{\norms{K}^2}{2\beta_0 k}\norms{x^0 - x^{\star}}^2 + \frac{\beta_0}{k+1}(\norms{\dot{y}} + M_g)^2.
\end{equation}
\item[$\mathrm{(c)}$]
If $f^{*}$ is $M_{f^{*}}$-Lipschitz continuous on $\dom{f^{*}}$, then for \eqref{eq:dual_prob}, it holds that
\begin{equation}\label{eq:main_result1c}
D(\tilde{y}^k) - D^{\star} \leq \frac{\norms{K}^2}{2\beta_0 k}(\norms{x^0} + M_{f^{*}})^2 + \frac{\beta_0}{k+1}\norms{\dot{y} - y^{\star}}^2.
\end{equation}
\end{itemize}
\end{theorem}

\begin{proof}
First, since $\mu_f = \mu_{g^{*}} = 0$, $L_k = \frac{\norms{K}^2}{\beta_k}$, and $(1+\tau_k)\beta_k = \beta_{k-1}$, the two conditions of \eqref{eq:param_cond} respectively reduce to 
\begin{equation*}
(1-\tau_k)\tau_{k-1}^2 \geq (1+\tau_k)\tau_k^2 \quad\text{and}\quad (1 + \tau_k)\tau_k^2 \geq \tau_{k-1}^2(1-\tau_k).
\end{equation*}
These conditions hold if $\tau_k^3 + \tau_k^2 + \tau_{k-1}^2\tau_k - \tau_{k-1}^2 = 0$.
We first choose $\tau_0 := 1$, and update $\tau_k$ by solving the cubic equation $\tau^3 + \tau + \tau_{k-1}^2\tau - \tau_{k-1}^2 = 0$ for $k\geq 1$. Note that this equation has a unique positive real solution $\tau_k \in (0, 1)$ due to Lemma~\ref{le:tech_lemma}(b).
Moreover, we have $\prod_{i=1}^k(1-\tau_i) \leq \frac{1}{k+1}$ and $\beta_k \leq \frac{2\beta_0}{k+2}$.

Next, by induction, \eqref{eq:key_est1_ncvx} leads to $\Vc_k(x) \leq \left[\prod_{i=1}^{k-1}(1-\tau_i)\right]\Vc_1(x) \leq \frac{1}{k}\Vc_1(x)$, where we have used $\prod_{i=1}^{k-1}(1-\tau_i) \leq \frac{1}{k}$ from Lemma~\ref{le:tech_lemma}(b).
However, from \eqref{eq:key_est00_proof1} in the proof of Lemma~\ref{le:maintain_gap_reduction1} and $\tau_0 = 1$, we have $\Vc_1(x) \leq (1-\tau_0)\Vc_0(x) + \frac{L_0\tau_0^2}{2}\norms{\hat{x}^0 - x}^2 = \frac{\norms{K}^2}{2\beta_0}\norms{x^0 - x}^2$.
Hence, we eventually obtain
\begin{equation}\label{eq:proof_th1_01}
\Vc_k(x) \leq \frac{\norms{K}^2}{2\beta_0k}\norms{x^0 - x}^2.
\end{equation}
Using \eqref{eq:key_est01}  from Lemma~\ref{le:basic_properties} and $\beta_{k-1} \leq  \frac{2\beta_0}{k+1}$ from Lemma~\ref{le:tech_lemma}(b), we get
\begin{equation*}
\begin{array}{lcl}
\Lc(x^k, y) - \Lc(x, \tilde{y}^k) & \overset{\tiny\eqref{eq:key_est01}}{\leq} & F_{\beta_{k-1}}(x^k, \dot{y}) - \Lc(x,\tilde{y}^k) + \frac{\beta_{k-1}}{2}\norms{y- \dot{y}}^2 \vspace{1ex}\\
& \overset{\tiny\eqref{eq:lyapunov_func1}}{\leq} & \Vc_k(x) + \frac{\beta_{k-1}}{2}\norms{y - \dot{y}}^2 \vspace{1ex}\\
& \overset{\tiny\eqref{eq:proof_th1_01}}{\leq} & \frac{\norms{K}^2}{2\beta_0k}\norms{x^0 - x}^2 + \frac{\beta_0}{k+1}\norms{\dot{y} - y}^2.
\end{array}
\end{equation*}
Taking the supreme  over $\Xc$ and $\Yc$ both sides of the last estimate and using \eqref{eq:gap_func}, we obtain \eqref{eq:main_result1}.

Now, since $F_{\beta_{k-1}}(x^k,\dot{y}) - F^{\star} \overset{\tiny\eqref{eq:saddle_point}}{\leq} F_{\beta_{k-1}}(x^k,\dot{y}) - \Lc(x^{\star}, \tilde{y}^k) \overset{\tiny\eqref{eq:lyapunov_func1}}{\leq} \Vc_k(x^{\star})$, combining this inequality and \eqref{eq:proof_th1_01}, we get
\begin{equation*}
F_{\beta_{k-1}}(x^k,\dot{y}) - F^{\star} \leq  \frac{\norms{K}^2}{2\beta_0k}\norms{x^0 - x^{\star}}^2.
\end{equation*}
On the other hand, since $g$ is $M_g$-Lipschitz continuous, we have 
\begin{equation*}
\sup\set{ \norms{y-\dot{y}} : y\in\partial{g}(Kx^k)} \leq \norms{\dot{y}} + \sup\sets{\norms{y} : \norms{y} \leq M_g}  = \norms{\dot{y}} + M_g.
\end{equation*}
Hence, by \eqref{eq:key_bound_apdx3} of Lemma~\ref{le:g_properties}, we have $F(x^k) \leq F_{\beta_{k-1}}(x^k, \dot{y}) + \frac{\beta_{k-1}}{2}(\norms{\dot{y}} + M_g)^2 \leq F_{\beta_{k-1}}(x^k, \dot{y}) + \frac{\beta_0}{k+1}(\norms{\dot{y}} + M_g)^2$.
Combining both estimates, we obtain \eqref{eq:main_result1b}.

Finally, using \eqref{eq:key_est01}, we have
\begin{equation*}
\begin{array}{lcl}
D(\tilde{y}^k) - D^{\star} & \leq &  F_{\beta_{k-1}}(x^k, \dot{y}) - \Lc(\tilde{x}^k, \tilde{y}^k) + \frac{\beta_{k-1}}{2}\norms{\dot{y} - y^{\star}}^2 \vspace{1ex}\\
&\overset{\tiny\eqref{eq:lyapunov_func1}}{\leq} &  \Vc_k(\tilde{x}^k) + \frac{\beta_0}{k+1}\norms{ \dot{y} - y^{\star}}^2 \vspace{1ex}\\
&\overset{\tiny\eqref{eq:proof_th1_01}}{\leq} &   \frac{\norms{K}^2}{2\beta_0k} (\norms{x^0} + M_{f^{*}})^2 + \frac{\beta_0}{k+1}\norms{\dot{y} - y^{\star}}^2,
\end{array}
\end{equation*}
which proves \eqref{eq:main_result1c}. 
Here, since $\tilde{x}^k \in \partial{f^{*}}(-K^{\top}\tilde{y}^k)$, we have $\norms{x^0 - \tilde{x}^k} \leq \norms{\tilde{x}^k} + \norms{x^0} \leq M_{f^{*}} + \norms{x^0}$, which has been used in the last inequality.
\Eproof
\end{proof}

\beforesubsec
\subsection{\textbf{Case 2: $f$ is strongly convex and $g^{*}$ is convex ($\mu_f > 0$ and $\mu_{g^{*}} = 0$)}}\label{subsect:f_or_gstar_strongly_convex}
\aftersubsec
Next, we consider the case when only $f$ or $g^{*}$ is strongly convex.
Without loss of generality, we assume that $f$ is strongly convex with a strong convexity parameter $\mu_f > 0$, but $g^{*}$ is only convex with $\mu_{g^{*}} = 0$.
Otherwise, we switch the role of $f$ and $g^{*}$ in Algorithm~\ref{alg:A1}.

The following theorem establishes an optimal $\BigO{1/k^2}$ convergence rate (up to a constant factor) of Algorithm~\ref{alg:A1} in this case (i.e. $\mu_f > 0$ and $\mu_{g^{*}} = 0$).

\begin{theorem}\label{th:convergence_1a}
Suppose that Assumption~\ref{as:A1} holds and that $f$ is strongly convex with a convexity parameter $\mu_f > 0$, but $g^{*}$ is just convex $($i.e. $\mu_{g^{*}} = 0$$)$.
Let us choose $\tau_0 := 1$ and $\beta_0  \geq \frac{0.382\norms{K}^2}{\mu_f}$.
Let $\sets{(x^k, \tilde{y}^k)}$ be generated by Algorithm~\ref{alg:A1} using the update rule $\tau_{k+1} := \frac{\tau_k}{2}\big(\sqrt{\tau_k^2 + 4} - \tau_k\big)$ for $\tau_k$.
Then, we have:
\begin{itemize}
\item[$\mathrm{(a)}$]
The gap function $\Gc_{\Xc\times\Yc}$ be defined by \eqref{eq:gap_func} satisfies
\begin{equation}\label{eq:main_est300}
\Gc_{\Xc\times\Yc}(x^k, \tilde{y}^k) \leq \frac{2\norms{K}^2}{\beta_0(k+1)^2}\sup_{x\in\Xc}\norms{x^0 -x}^2 + \frac{10\beta_0}{(k+3)^2} \sup_{y\in\Yc}\norms{\dot{y} - y}^2.
\end{equation}

\item[$\mathrm{(b)}$]
If $g$ is  $M_g$-Lipschitz continuous on $\dom{g}$, then for \eqref{eq:primal_prob}, it holds that
\begin{equation}\label{eq:main_est300b}
F(x^k) - F^{\star} \leq \frac{2\norms{K}^2\norms{x^0 - x^{\star}}^2}{\beta_0(k+1)^2} + \frac{10\beta_0(\norms{\dot{y}} + M_g)^2}{(k+3)^2}.
\end{equation}

\item[$\mathrm{(c)}$]
If $f^{*}$ is $M_{f^{*}}$-Lipschitz continuous on $\dom{f^{*}}$, then for \eqref{eq:dual_prob}, it holds that
\begin{equation}\label{eq:main_est300c}
D(\tilde{y}^k) - D^{\star} \leq \frac{2\norms{K}^2(\norms{x^0} + M_{f^{*}})^2}{\beta_0(k+1)^2} + \frac{10\beta_0\norms{\dot{y} - y^{\star}}^2}{(k+3)^2}.
\end{equation}
\end{itemize}
\end{theorem}

\begin{proof}
Since $\mu_{g^{*}} = 0$ and $\tau_k^2 = \tau_{k-1}^2(1 - \tau_k)$ by the update rule of $\tau_k$, $\beta_k = \frac{\beta_{k-1}}{1+\tau_k}$, and $L_k = \frac{\norms{K}^2}{\beta_k}$, 
the first condition of \eqref{eq:param_cond} is equivalent to $\beta_{k-1} \geq \frac{\norms{K}^2\tau_k^2}{\mu_f}$.
However, since $\beta_{k-1}  \geq \frac{\beta_0\tau_{k}^2}{\tau_1^2}$ due to Lemma~\ref{le:tech_lemma}(a), and $\tau_1 = 0.6180$, $\beta_{k-1} \geq \frac{\norms{K}^2\tau_k^2}{\mu_f}$ holds if $\beta_0 \geq \frac{\norms{K}^2\tau_1^2}{\mu_f} = \frac{0.382\norms{K}^2}{\mu_f}$.
Thus we can choose $\beta_0 \geq  \frac{0.382\norms{K}^2}{\mu_f}$ to guarantee the  first condition of \eqref{eq:param_cond}.

Similarly, using $m_k = \frac{L_k + \mu_f}{L_{k-1} + \mu_f} \geq 1$, the second condition of \eqref{eq:param_cond} is equivalent to 
$m_k^2\tau_k^2 + m_k\tau_k\tau_{k-1}^2 \geq \frac{L_k\tau_{k-1}^2}{L_{k-1} + \mu_f}$.
Since $\frac{L_k}{L_{k-1} + \mu_f} \leq m_k$, the last condition holds if $m_k^2\tau_k^2 + m_k\tau_k\tau_{k-1}^2 \geq m_k\tau_{k-1}^2$.
Using again $\tau_k^2 = \tau_{k-1}^2(1-\tau_k)$, this condition becomes $m_k\tau_k^2 \geq \tau_{k-1}^2(1-\tau_k) = \tau_k^2$.
This always holds true since $m_k \geq 1$.
Therefore,  the second condition of \eqref{eq:param_cond} is satisfied.

As a result, we have the recursive estimate \eqref{eq:key_est1_ncvx}, i.e.:
\begin{equation}\label{eq:recursive_form}
\Vc_{k+1}(x) \leq (1-\tau_k)\Vc_{k}(x).
\end{equation}
From \eqref{eq:recursive_form}, Lemma~\ref{le:tech_lemma}(a), \eqref{eq:key_est00_proof1}, and noting that $\tilde{x}^0 := x^0$ and $\tau_0 := 1$, we have 
\begin{equation*}
\arraycolsep=0.2em
\begin{array}{lcl}
\Vc_k(x) & \leq & \big[\prod_{i=1}^{k-1}(1-\tau_i)\big]\Vc_1(x) = \Theta_{1,k-1}\Vc_1(x) \leq \frac{4}{(k+1)^2}\Rc_0(x).
\end{array}
\end{equation*}
where $\Rc_0(x) :=  \frac{\Vert K\Vert^2}{2\beta_0}\Vert x^0 - x\Vert^2$.
Similar to the proof of Theorem~\ref{th:convergence1}, using the last inequality and $\beta_{k-1} \leq \frac{4\beta_0\tau_0^2}{\tau_2^2[\tau_0(k+1) + 2]^2} \leq \frac{20\beta_0}{(k+3)^2}$ from Lemma~\ref{le:tech_lemma}(a), we obtain the bounds \eqref{eq:main_est300}, \eqref{eq:main_est300b}, and  \eqref{eq:main_est300c} respectively.
\Eproof
\end{proof}

\begin{remark}\label{re:compared}
The variant of Algorithm~\ref{alg:A1} in Theorem~\ref{th:convergence_1a} is completely different from \cite[Algorithm 2]{TranDinh2015b} and \cite[Algorithm 2]{tran2019non}, where it requires only one $\prox_{f/L_k}(\cdot)$ as opposed to two proximal operations  of $f$ as in \cite{TranDinh2015b,tran2019non}.
\end{remark}

\beforesubsec
\subsection{\textbf{Case 3: Both $f$ and $g^{*}$ are strongly convex ($\mu_f > 0$  and $\mu_{g^{*}} > 0$)}}\label{sec:both_f_gstar_scvx}
\aftersubsec
Finally, we assume that both $f$ and $g^{*}$ are strongly convex with strong convexity parameters $\mu_f > 0$ and $\mu_{g^{*}} > 0$, respectively.
Then, the following theorem proves the optimal linear rate (up to a constant factor) of Algorithm~\ref{alg:A1}.

\begin{theorem}\label{th:convergence_1b}
Suppose that Assumption~\ref{as:A1} holds and both $f$ and $g^{*}$ in \eqref{eq:saddle_point_prob} are strongly convex with $\mu_f > 0$ and $\mu_{g^{*}} > 0$, respectively.
Let $\sets{(x^k, \tilde{y}^k)}$ be generated by Algorithm~\ref{alg:A1} using $\tau_k := \tau = \frac{1}{\sqrt{1 + \kappa_F}}\in (0, 1)$ and $\beta_0 > 0$, where $\kappa_F := \frac{\norms{K}^2}{\mu_f\mu_{g^{*}}}$.
Then, the following statements hold:
\begin{itemize}
\item[$\mathrm{(a)}$]
The gap function $\Gc_{\Xc\times\Yc}$ defined by \eqref{eq:gap_func} satisfies
\begin{equation}\label{eq:key_est40}
 \Gc_{\Xc\times\Yc}(x^k,\tilde{y}^k) \leq  (1-\tau)^k\bar{\Rc}_p + \frac{1}{2(1+\tau)^k}\sup_{y\in\Yc}\norms{\dot{y} - y}^2,
 \vspace{-0.5ex}
\end{equation}
where $\bar{\Rc}_p := {\displaystyle\sup_{x\in\Xc}}\Big\{ (1-\tau)\big[\Fc_{\beta_0}(x^0) - \Lc(x, \tilde{y}^0)\big] + \frac{\norms{K}^2\tau^2}{2(\mu_{g^{*}} + \beta_{0})} \norms{x^0 - x}^2 \Big\}$.

\item[$\mathrm{(b)}$]
If $g$ is $M_g$-Lipschitz continuous on $\dom{g}$, then for \eqref{eq:primal_prob}, it holds that
\begin{equation}\label{eq:key_est40b}
F(x^k) - F^{\star} \leq (1-\tau)^k\bar{\Rc}_p^{\star} \ +  \ \frac{\beta_0}{2(1+\tau)^k} \left(\norms{\dot{y}} + M_g\right)^2,
\end{equation}
where $\bar{\Rc}_p^{\star} := (1-\tau)\big[ \Fc_{\beta_0}(x^0) - \Lc(x^{\star}, \tilde{y}^0) \big] +  \frac{\norms{K}^2\tau^2}{2(\mu_{g^{*}} + \beta_{0})}\norms{x^0 - x^{\star}}^2$.

\item[$\mathrm{(d)}$]
If $f^{*}$ is $M_{f^{*}}$-Lipschitz continuous on $\dom{f^{*}}$, then for \eqref{eq:dual_prob}, it holds that
\begin{equation}\label{eq:key_est40c}
D(\tilde{y}^k) - D^{\star} \leq (1-\tau)^k\bar{\Rc}_d^{\star} \ + \ \frac{\beta_0\norms{\dot{y} - y^{\star}}^2}{2(1+\tau)^k},
\end{equation}
where $\bar{\Rc}_d^{\star} := (1-\tau)\big[\Fc_{\beta_0}(x^0) - D(\tilde{y}^0)\big] +  \frac{\norms{K}^2\tau^2}{2(\mu_{g^{*}} + \beta_{0})}\left(\norms{x^0} + M_{f^{*}}\right)^2$.
\end{itemize}
\end{theorem}

\begin{proof}
Since $\tau_k = \tau = \frac{1}{\sqrt{1+\kappa_F}} = \sqrt{\frac{\mu_f\mu_{g^{*}}}{\norms{K}^2 + \mu_f\mu_{g^{*}}}} \in (0, 1)$ and $\beta_{k-1} = (1+\tau)\beta_k$, after a few elementary calculations, we can show that the first condition of \eqref{eq:param_cond} automatically holds.
The second condition of \eqref{eq:param_cond} is equivalent to $m_k\tau + m_k^2 \geq \frac{L_k}{L_{k-1} + \mu_f}$.
Since $m_k \geq \frac{L_k}{L_{k-1} + \mu_f}$, this condition holds if $m_k\tau + m_k^2 \geq m_k$, which is equivalent to $\tau + m_k \geq 1$. This obviously holds true since $\tau > 0$ and $m_k \geq 1$.

From \eqref{eq:key_est1_ncvx} of Lemma~\ref{le:maintain_gap_reduction1}, we have $\Vc_{k+1}(x) \leq (1-\tau)\Vc_{k}(x)$. 
Therefore, by induction and using again \eqref{eq:key_est00_proof1}, we get
\begin{equation}\label{eq:recursive_form}
\Vc_k(x) \leq  (1-\tau)^k\Vc_1(x) \leq  (1-\tau)^k\bar{\Rc}_0(x).
\end{equation}
where $\bar{\Rc}_p(x) :=  (1-\tau)\big[ \Fc_{\beta_0}(x^0) - \Lc(x,\tilde{y}^0)\big] + \frac{\norms{K}^2\tau^2}{2(\mu_{g^{*}} + \beta_{0})}\norms{x^0 - x}^2$.

Now, since $\beta_{k-1} = \frac{\beta_0}{(1+\tau)^k}$ due to the update rule of $\beta_k$, by \eqref{eq:key_est01}, we have
\begin{equation*}
\Lc(x^k, y) - \Lc(x, \tilde{y}^k) \leq  \Vc_k(x) + \frac{\beta_{k-1}}{2}\norms{\dot{y} - y}^2  \leq  (1-\tau)^k\bar{\Rc}_p(x) + \frac{\norms{\dot{y} - y}^2}{2(1+\tau)^k}.
\end{equation*}
This implies \eqref{eq:key_est40}.
The estimates \eqref{eq:key_est40b} and \eqref{eq:key_est40c} can be proved similarly as in Theorem~\ref{th:convergence1}, and we omit the details here.
\Eproof
\end{proof}

\begin{remark}\label{re:cond_number}
Since $g^{*}$ is $\mu_{g^{*}}$-strongly convex, it is well-known that $g\circ K$ is $\frac{\norms{K}^2}{\mu_{g^{*}}}$-smooth.
Hence, $\kappa_F := \frac{\norms{K}^2}{\mu_f\mu_{g^{*}}}$ is the condition number of $F$ in \eqref{eq:primal_prob}.
Theorem shows that Algorithm~\ref{alg:A1} can achieve a $\big(1 - \frac{1}{(1+\sqrt{2})\sqrt{\kappa_F}}\big)^k$-linear convergence rate.
Consequently, it also achieves $\BigO{ \sqrt{\kappa_F}\log\left(\frac{1}{\varepsilon}\right)}$ oracle complexity to obtain an $\varepsilon$-primal-dual solution $(x^k,\tilde{y}^k)$.
This linear rate and complexity are optimal (up to a constant factor) under the given assumptions in Theorem~\ref{th:convergence_1b}.
However, Algorithm~\ref{alg:A1} is very different from existing accelerated proximal gradient methods, e.g., \cite{Beck2009,Nesterov2007,tseng2008accelerated} for solving \eqref{eq:primal_prob} since our method uses the proximal operator of $g^{*}$ (and therefore, the proximal operator of $g$) instead of the gradient of $g$ as in \cite{Beck2009,Nesterov2007,tseng2008accelerated}.
\end{remark}

\begin{remark}\label{re:optimal_rate}
The $\BigO{1/k}$, $\BigO{1/k^2}$, and linear convergence rates in Theorem~\ref{th:convergence1}, \ref{th:convergence_1a}, and \ref{th:convergence_1b}, respectively are already optimal (up to a constant factor) under given assumptions as discussed, e.g., in \cite{ouyang2018lower,tran2019non}.
The primal convergence rate on $\sets{F(x^k) - F^{\star}}$ has been proved in \cite[Theorem~4]{TranDinh2015b}, but only for the case $\BigO{1/k}$.
The convergence rates on $\sets{\Gc_{\Xc\times\Yc}(x^k,\tilde{y}^k)}$ and $\sets{D(\tilde{y}^k) - D^{\star}}$ are new.
Moreover, the convergence of the primal sequence is on the last iterate $x^k$, while the convergence of the dual sequence is on the averaging iterate $\tilde{y}^k$.
\end{remark}

\beforesec
\section{Numerical Experiments}\label{sec:num_exp}
\aftersec
In this section, we provide four numerical experiments to verify the theoretical convergence aspects and the performance of Algorithm~\ref{alg:A1}. 
Our algorithm is implemented in Matlab R.2019b running on a MacBook Laptop with 2.8GHz Quad-Core Intel Core i7 and 16GB RAM.
We also compare our method with Nesterov's smoothing technique in \cite{Nesterov2005c} as a baseline.
We emphasize that our experiments bellow follow exactly the parameter update rules as stated in Theorem~\ref{th:convergence1} and Theorem~\ref{th:convergence_1a} without any parameter tuning trick.
To further improve practical performance of Algorithm~\ref{alg:A1}, one can exploit the restarting strategy in \cite{TranDinh2015b}, where its theoretical guarantee is established in \cite{tran2018adaptive}.

The nonsmooth and convex optimization problem we use for our experiments is the following representative model:
\begin{equation}\label{eq:LAD}
\min_{x\in\R^p}\Big\{ F(x) := \Vert Kx - b\Vert_2 + \lambda\norms{x}_1 + \frac{\rho}{2}\norms{x}_2^2 \Big\},
\end{equation}
where $K\in\R^{n\times p}$ is a given matrix, $b\in\R^n$ is also given, and $\lambda > 0$ and $\rho \geq 0$ are two given regularization parameters.
The norm $\norms{\cdot}_2$ is the $\ell_2$-norm (or Euclidean norm).
If $\rho = 0$, then \eqref{eq:LAD} reduces to the square-root LASSO model proposed in \cite{Belloni2011}.
If $\rho > 0$, then \eqref{eq:LAD} becomes a square-root regression problem with elastic net regularizer similar to \cite{zou2005regularization}.
Clearly, if we define $g(y) := \norms{y - b}_2$ and $f(x) := \lambda\norms{x}_1 + \frac{\rho}{2}\norms{x}_2^2$, then \eqref{eq:LAD} can be rewritten into \eqref{eq:primal_prob}.

To generate the input data for our experiments, we first generate $K$ from standard i.i.d. Gaussian distributions with either uncorrelated or $50\%$ correlated columns.
Then, we generate an observed vector $b$ as $b := Kx^{\natural} + \mathcal{N}(0,\sigma)$, where $x^{\natural}$ is a predefined sparse vector and $\mathcal{N}(0, \sigma)$ stands for standard Gaussian noise with variance $\sigma = 0.05$.
The regularization parameter $\lambda$ to promote sparsity is chosen as suggested in \cite{Belloni2011}, and the parameter $\rho$ is set to $\rho := 0.1$.
We first fix the size of problem at $p := 1000$ and $n := 350$ and choose the number of nonzero entries of $x^{\natural}$ to be $s := 100$.
Then, for each experiment, we generate $30$ instances of the same size but with different input data $(K, b)$.

For Nesterov's smoothing method, by following \cite{Nesterov2005c}, we smooth $g$ as 
\begin{equation*}
g_{\gamma}(y) := \max_{v\in\R^n}\Big\{ \iprods{y - b, v} - \frac{\gamma}{2}\norms{v}^2: \norms{v}_2 \leq 1 \Big\},
\end{equation*}
where $\gamma > 0$ is a smoothness parameter.
In order to correctly choose $\gamma$ for Nesterov's smoothing method, we first solve \eqref{eq:LAD} with CVX \cite{Grant2004} using Mosek with high precision to get a high accurate solution $x^{\star}$ of \eqref{eq:LAD}. 
Then, we set $\gamma^{*} := \frac{\sqrt{2}\norms{K}\norms{x^0 - x^{\star}}}{k_{\max}\sqrt{D_{\mathcal{V}}}}$ by minimizing its theoretical bound from \cite{Nesterov2005c} w.r.t. $\gamma > 0$, where $D_{\mathcal{V}} := \frac{1}{2}$ is the prox-diameter of the unit $\ell_2$-norm ball, and $k_{\max}$ is the maximum number of iterations.
For Algorithm~\ref{alg:A1}, using \eqref{eq:main_result1b} and $\dot{y} := 0$, we can set $\beta_0 = \beta^{*} := \frac{\norms{K}\norms{x^0 - x^{*}}}{M_g}$ by minimizing the right-hand side of \eqref{eq:main_result1b} w.r.t. $\beta_0 > 0$, where $M_g := 1$.
We choose $k_{\max} := 5000$ for all experiments.
To see the effect of the smoothness parameters $\gamma$ and $\beta_0$ on the performance of both algorithms, we also consider two variants by increasing or decreasing these parameters $10$ times, respectively. More specifically, we set them as follows.
\begin{itemize}
\itemsep=0.1em
\item For Nesterov's smoothing scheme, we consider two additional variants by setting $\gamma := 10\gamma^{*}$ and $\gamma = 0.1\gamma^{*}$, respectively.
\item For Algorithm~\ref{alg:A1}, we also consider two other variants with $\beta_0 := 10\beta^{*}$ and $\beta_0 := 0.1\beta^{*}$, respectively.
\end{itemize}
We first conduct two different experiments for the square-root LASSO model (i.e. setting $\rho := 0$ in \eqref{eq:LAD}).
In this case, the underlying optimization problem is non-strongly convex and fully nonsmooth. 
\begin{itemize}
\itemsep=0.1em
\item \textit{Experiment 1:}
We test Algorithm \ref{alg:A1} (abbreviated by \texttt{Alg. 1}) and Nesterov's smoothing method (abbreviated by \texttt{Nes. Alg.}) on $30$ problem instances with uncorrelated columns of $K$.
Since both algorithms have essentially the same per-iteration complexity, we report the relative primal objective residual $\frac{F(x^k) - F(x^{\star})}{\max\sets{1, \vert F(x^{\star})\vert}}$ against the number of iterations.
\item \textit{Experiment 2:} We conduct the same test on another set of $30$ problem instances, but using $50\%$ correlated columns in the input matrix $K$.
\end{itemize}
The results of both experiments are depicted in Figure~\ref{fig:sqrtLASSO_exp1}, where the left plot is for \textit{Experiment 1} and the right plot is for \textit{Experiment 2}.
The solid line of each curve presents the mean over $30$ problem samples, and the corresponding shaded area represents the sample variance of $30$ problem samples (i.e. the area from the lowest and the  highest deviation to the mean).
\begin{figure}[ht!]
\vspace{-2ex}
\centering
\includegraphics[width=1\textwidth]{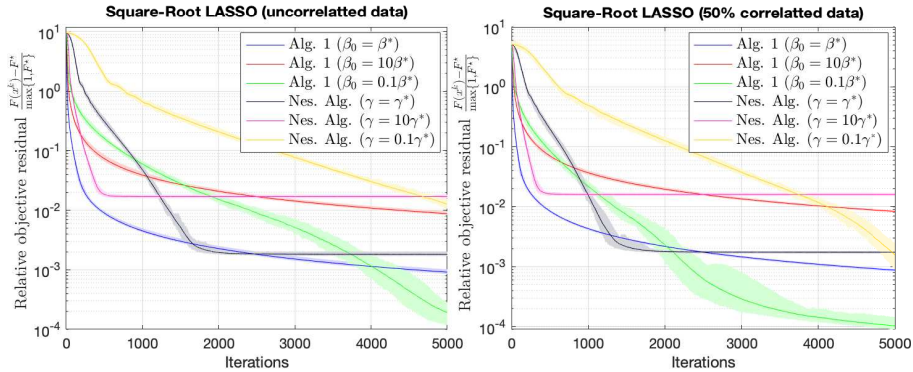}
\vspace{-2ex}
\caption{The convergence behavior of Algorithm~\ref{alg:A1} and Nesterov's smoothing scheme on $30$ problem instances of \eqref{eq:LAD} (the non-strongly convex case). Left plot: uncorrelated columns in $K$, and Right plot: $50\%$ correlated columns in $K$.}
\label{fig:sqrtLASSO_exp1}
\vspace{-2ex}
\end{figure} 

From Figure \ref{fig:sqrtLASSO_exp1}, we can observe that, with the choice $\beta_0 := \beta^{*}$ and $\gamma := \gamma^{*}$ as suggested by the theory, both algorithms perform best compared to other smaller or larger values of these parameters. 
We also see that Algorithm~\ref{alg:A1} outperforms Nesterov's smoothing scheme in both experiments.
If $\beta_0$ (respectively, $\gamma$) is large, then both algorithms make good progress in early iterations, but become saturated at a given objective value in the last iterations.
Alternatively, if $\beta_0$ (respectively, $\gamma$) is small, then both algorithms perform worse  in early iterations, but further decrease the objective value when the number of iterations is increasing.
This behavior also confirms the theoretical results stated in Theorem~\ref{th:convergence1} and in \cite{Nesterov2005c}.
In fact, if $\beta_0$ (or $\gamma$) is small, then the algorithmic stepsize is small.
Hence, the algorithm makes slow progress at early iterations, but it better approximates the nonsmooth function $g$, leading to more accurate approximation from $F(x^k)$ to $F(x^{\star})$.
In contrast, if $\beta_0$ (or $\gamma$) is large, then we have a large stepsize and therefore a faster convergence rate in early iterations.
However, the smoothed approximation is less accurate.

In order to test the strongly convex case in Theorem~\ref{th:convergence_1a}, we conduct two additional experiments on \eqref{eq:LAD} with $\rho := 0.1$.
In this case, problem \eqref{eq:LAD} is strongly convex with $\mu_f := 0.1$.
Since \cite{Nesterov2005c} does not directly handle the strongly convex case, we only compare two variants of Algorithm~\ref{alg:A1} stated in Theorem~\ref{th:convergence1} (\texttt{Alg. 1}) and Theorem~\ref{th:convergence_1a} (\texttt{Alg. 1b}), respectively.
We set $\beta_0 := \frac{0.382\norms{K}^2}{\mu_f}$ in \texttt{Alg. 1b)} as suggested by Theorem~\ref{th:convergence_1a}.
We consider two experiments as follows:
\begin{itemize}
\itemsep=0.1em
\item \textit{Experiment 3:}
Test two variants of Algorithm \ref{alg:A1} on a collection of $30$ problem instances with uncorrelated columns of $K$.
\item \textit{Experiment 4:} 
Conduct the same test on another set of $30$ problem instances, but using $50\%$ correlated columns in $K$.
\end{itemize}
The results of both variants of Algorithm~\ref{alg:A1} are reported in Figure \ref{fig:sqrtLASSO_exp2}, where the left plot is for \textit{Experiment 3} and the right plot is for \textit{Experiment 4}.
\begin{figure}[ht!]
\vspace{-3ex}
\centering
\includegraphics[width=1\textwidth]{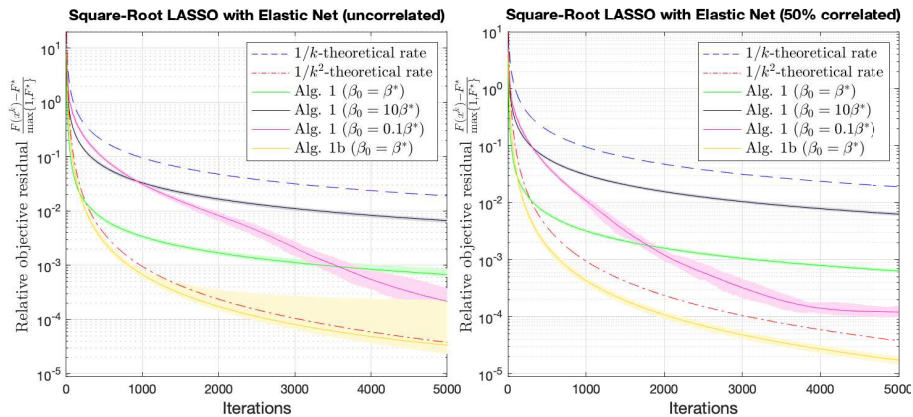}
\vspace{-2ex}
\caption{The convergence behavior of the two variants of Algorithm~\ref{alg:A1} on a collection of $30$ problem instances of \eqref{eq:LAD} (the strongly convex case). Left plot: uncorrelated columns in $K$, and Right plot: $50\%$ correlated columns in $K$.}
\label{fig:sqrtLASSO_exp2}
\vspace{-3ex}
\end{figure} 

Clearly, as shown in Figure \ref{fig:sqrtLASSO_exp2},  \texttt{Alg. 1b} (i.e. corresponding to Theorem~\ref{th:convergence_1a}) highly outperforms  \texttt{Alg. 1} (corresponding to Theorem~\ref{th:convergence1}). 
\texttt{Alg. 1} matches well the $\BigO{1/k}$ convergence rate as stated in Theorem~\ref{th:convergence1}, while \texttt{Alg. 1b} shows its $\BigO{1/k^2}$ convergence rate as indicated by Theorem~\ref{th:convergence_1a}.
Note that since $g^{*}$ in \eqref{eq:LAD} is non-strongly convex, we omit testing the result of Theorem~\ref{th:convergence_1b}.
This case is rather well studied in the literature, see, e.g., \cite{Chambolle2011}.

\beforesec
\section{Concluding Remarks}\label{sec:concl}
\aftersec
We have developed a new variant of ASGRD introduced in \cite[Algorithm 1]{TranDinh2015b}, Algorithm~\ref{alg:A1}, that unifies three different settings: general convexity, strong convexity, and strong convexity and smoothness.
We have proved the convergence of Algorithm~\ref{alg:A1} for three settings on three convergence criteria: gap function, primal objective residual, and dual objective residual.
Our convergence rates in all cases are optimal up to a constant factor and the convergence rates of the primal sequence is on the last iterate.
Our preliminary numerical experiments have shown  that the theoretical convergence rates of Algorithm~\ref{alg:A1} match well the actual rates observed in practice. 
The proposed algorithm can be easily extended to solve composite convex problems with three or multi-objective terms.
It can also be customized to solve other models, including general linear and nonlinear constrained convex problems as discussed in \cite{sabach2020faster,TranDinh2015b,tran2019non}.

\vskip 2mm
\noindent{\bf Acknowledgments.}
This work is partly supported by the Office of Naval Research under Grant No.  ONR-N00014-20-1-2088 (2020 - 2023), and the Nafosted Vietnam, Grant No. 101.01-2020.06 (2020 - 2022).

\appendix

\beforesec
\section{Appendix 1: Technical Lemmas}\label{apdx:basic}
\aftersec
We need the following technical lemmas for our convergence analysis in the main text.

\begin{lemma}\label{le:g_properties}$($\cite[Lemma 10]{TranDinh2015b}$)$
Given $\beta > 0$, $\dot{y}\in\R^n$, and a proper, closed, and convex function $g : \R^n \to\Rext$ with its Fenchel conjugate $g^{*}$, we define
\begin{equation}\label{eq:smoothed_varphi0}
g_{\beta}(u, \dot{y}) := \max_{y\in\R^n}\set{\iprods{u, y} -  g^{*}(y)  - \tfrac{\beta}{2}\norms{y - \dot{y}}^2}.
\end{equation}
Let $y^{*}_{\beta}(u, \dot{y})$ be the unique solution of \eqref{eq:smoothed_varphi0}.
Then, the following statements hold:

\begin{itemize}
\item[$(a)$] $g_{\beta}(\cdot,\dot{y})$ is convex w.r.t. $u$ on $\dom{g}$ and $\frac{1}{\beta + \mu_{g^{*}}}$-smooth w.r.t. $u$ on $\dom{g}$, where $\nabla_u{g_{\beta}}(u,\dot{y}) = \prox_{g^{*}/\beta}(\dot{y} + \frac{1}{\beta}u)$.
Moreover, for any $u, \hat{u} \in\dom{g}$, we have
\begin{equation}\label{eq:key_bound_apdx1}
g_{\beta}(\hat{u},\dot{y}) + \iprods{\nabla{g}_{\beta}(\hat{u},\dot{y}), u - \hat{u}} \leq g_{\beta}(u,\dot{y}) - \frac{\beta + \mu_{g^{*}}}{2}\Vert \nabla_u{g_{\beta}}(\hat{u},\dot{y}) - \nabla_u{g_{\beta}}(u,\dot{y})\Vert^2.
\end{equation}

\item[$(b)$] For any $\beta > 0$, $\dot{y}\in\R^n$, and $u\in\dom{g}$, we have
\begin{equation}\label{eq:key_bound_apdx3}
\hspace{-1ex}
\begin{array}{lcl}
g_{\beta}(u,\dot{y}) \leq g(u) \leq g_{\beta}(u,\dot{y}) + \frac{\beta}{2} [ D_{g}(\dot{y})]^2, \ \text{where} \ D_{g}(\dot{y}) := \sup_{y\in\partial{g(u)}} \norm{y - \dot{y}}.
\end{array}
\hspace{-1ex}
\end{equation}

\item[$(c)$] 
For $u\in\dom{g}$ and $\dot{y}\in\R^n$, $g_{\beta}(u,\dot{y})$ is convex in $\beta$, and for all $\hat{\beta} \geq \beta > 0$, we have
\begin{equation}\label{eq:key_bound_apdx4}
g_{\beta}(u,\dot{y}) \leq g_{\hat{\beta}}(u,\dot{y}) +  \big(\tfrac{\hat{\beta} - \beta}{2}\big)\norms{\nabla_u{g_{\beta}}(u,\dot{y}) - \dot{y} }^2.
\hspace{-2.5ex}
\end{equation}

\item[$(d)$] For any $\beta > 0$, and $u, \hat{u}\in\dom{g}$, we have
\begin{equation}\label{eq:key_bound_apdx5}
\hspace{5ex}
\begin{array}{lcl}
g_{\beta}(u,\dot{y}) + \iprods{\nabla_u{g_{\beta}}(u, \dot{y}), \hat{u} - u} & \leq & \ell_{\beta}(\hat{u}, \dot{y}) - \frac{\beta}{2}\norms{\nabla_u{g_{\beta}}(u,\dot{y}) - \dot{y}}^2,
\end{array}
\end{equation}
where $\ell_{\beta}(\hat{u}, \dot{y}) := \iprods{\hat{u}, \nabla_u{g_{\beta}}(u,\dot{y}) } - g^{*}(\nabla_u{g_{\beta}}(u,\dot{y})) \leq g(\hat{u}) - \frac{\mu_{g^{*}}}{2}\norms{\nabla_u{g_{\beta}}(u,\dot{y})   - \nabla{g}(\hat{u})}^2$ for any $\nabla{g}(\hat{u}) \in \partial{g}(\hat{u})$.
\end{itemize}
\end{lemma}

\begin{lemma}\label{le:tech_lemma}
The following statements hold.
\begin{itemize}
\item[$\mathrm{(a)}$] 
Let $\set{\tau_k}\subset (0, 1]$ be computed by $\tau_{k+1} := \frac{\tau_k}{2}\big[ (\tau_k^2 + 4)^{1/2} - \tau_k\big]$ for some $\tau_0 \in (0, 1]$.
Then, we have 
\begin{equation*}
 \tau_k^2 = (1-\tau_k)\tau_{k-1}^2, \quad \frac{1}{k + 1/\tau_0} \leq \tau_k < \frac{2}{k + 2/\tau_0}, \quad\text{and}\quad  \frac{1}{1 + \tau_{k-2}} \leq 1 - \tau_k \leq \frac{1}{1+\tau_{k-1}}.
\end{equation*}
Moreover, we also have 
\begin{equation*}
\begin{array}{ll}
& \Theta_{l,k} := \displaystyle\prod_{i=l}^k(1-\tau_i) = \dfrac{\tau_k^2}{\tau_{l-1}^2} \quad\text{for}\ 0\leq l\leq k, \qquad \Theta_{0,k} = \dfrac{(1-\tau_0)\tau_k^2}{\tau_0^2} \leq \dfrac{4(1-\tau_0)}{(\tau_0k+2)^2}, \vspace{1ex}\\
\text{and}\quad &\dfrac{\tau_{l+1}^2}{\tau_{k+2}^2}  \leq  \Gamma_{l,k} := \displaystyle\prod_{i=l}^k(1+\tau_i) \leq  \dfrac{\tau_l^2}{\tau_{k+1}^2} \quad\text{for} \ 0 \leq l \leq k.
\end{array}
\end{equation*}
If we update $\beta_k := \frac{\beta_{k-1}}{1+\tau_k}$ for a given $\beta_0 > 0$, then 
\begin{equation*}
\frac{4\beta_0\tau_0^2}{\tau_1^2[\tau_0(k+1) + 2]^2} \leq \frac{\beta_0\tau_{k+1}^2}{\tau_1^2} \leq \beta_k = \frac{\beta_0}{\Gamma_{1,k}}\leq \frac{\beta_0\tau_{k+2}^2}{\tau_2^2} \leq \frac{4\beta_0\tau_0^2}{\tau_2^2[\tau_0(k+2) + 2]^2}.
\end{equation*}
\item[$\mathrm{(b)}$] 
Let $\set{\tau_k}\subset (0, 1]$ be  computed by solving $\tau_k^3 + \tau_k^2 + \tau_{k-1}^2\tau_k - \tau_{k-1}^2 = 0$ for all $k\geq 1$ and $\tau_0 := 1$.
Then, we have $\frac{1}{k+1} \leq \tau_k \leq \frac{2}{k+2}$ and $\Theta_{1,k} := \prod_{i=1}^k(1-\tau_i) \leq \frac{1}{k+1}$.
Moreover, if we update $\beta_k := \frac{\beta_{k-1}}{1+\tau_k}$, then $\beta_k \leq \frac{2\beta_0}{k+2}$.
\end{itemize}
\end{lemma}

\begin{proof}
The first two relations of (a) have been proved, e.g., in \cite{TranDinh2012a}. 
Let us prove the last inequality of (a).
Since $\frac{1}{1+\tau_{k-2}} \leq 1-\tau_k$ is equivalent to $\tau_{k-2}(1-\tau_k) \geq \tau_k$.
Using $1- \tau_k = \frac{\tau_k^2}{\tau_{k-1}^2}$, we have $\tau_k\tau_{k-2} \geq \tau_{k-1}^2$.
Utilizing $\tau_k = \frac{\tau_{k-1}}{2}\big[(\tau_{k-1}^2 + 4)^2 - \tau_{k-1}\big]$, this condition is equivalent to $\tau_{k-2}^2 \geq \tau_{k-1}^2(1 + \tau_{k-2})$.
However, since $\tau_{k-1}^2 = (1-\tau_{k-1})\tau_{k-2}^2$, the last condition becomes $1 \geq (1-\tau_{k-1})(1+\tau_{k-2})$, or equivalently, $\tau_{k-1} \leq \tau_{k-2}$, which automatically holds.

To prove $1-\tau_k \leq \frac{1}{1 + \tau_{k-1}}$, we write it as $\tau_{k-1}(1-\tau_k) \leq \tau_{k}$.
Using again $\tau_k^2 = (1-\tau_k)\tau_{k-1}^2$, the last inequality is equivalent to $\tau_k \leq \tau_{k-1}$, which automatically holds.
The last statements of (a) is a consequence of $1-\tau_k = \frac{\tau_k^2}{\tau_{k-1}^2}$ and the previous relations.

(b) We consider the function $\varphi(\tau) := \tau^3 + \tau^2 + \tau_{k-1}^2\tau - \tau_{k-1}^2$.
Clearly, $\varphi(0) = -\tau_{k-1}^2 < 0$ and $\varphi(1) = 2 > 0$. Moreover, $\varphi'(\tau) = 3\tau^2 + 2\tau + \tau_{k-1}^2 > 0$ for $\tau \in [0, 1]$.
Hence, the cubic equation $\varphi(\tau) = 0$ has a unique solution $\tau_k \in (0, 1)$.
Therefore, $\sets{\tau_k}_{k\geq 0}$ is well-defined.

Next, since $\tau_k^3 + \tau_k^2 + \tau_k\tau_{k-1}^2 - \tau_{k-1}^2 = 0$ is equivalent to $\tau_{k-1}^2(1-\tau_k) = \tau_k^2(1+\tau_k)$, we have  $\tau_{k-1}^2(1-\tau_k) = \tau_k^2(1+\tau_k) \leq \frac{\tau_k^2}{1-\tau_k}$.
This inequality becomes $\tau_k \geq \frac{\tau_{k-1}}{1 + \tau_{k-1}}$. 
By induction and $\tau_0 = 1$, we can easily show that $\tau_k \geq \frac{1}{k+1}$.
On the other hand, $\tau_{k-1}^2(1-\tau_k) = \tau_k^2(1+\tau_k) \geq \tau_k^2$.
From this inequality, with a similar argument as in the proof of the statement (a), we can also easily show that $\tau_k \leq \frac{2}{k+2}$.
Hence, we have $\frac{1}{k+1} \leq \tau_k \leq \frac{2}{k+2}$ for all $k\geq 0$.

Finally, since $\tau_k \geq \frac{1}{k+1}$, we have $\prod_{i=1}^k(1-\tau_i) \leq \prod_{i=1}^k\left(1 - \frac{1}{i+1}\right) = \frac{1}{k+1}$.
Alternatively, $\prod_{i=1}^k(1+\tau_i) \geq \prod_{i=1}^k\left(1 + \frac{1}{i+1}\right) = \frac{k+2}{2}$.
However, since $\beta_k = \frac{\beta_{k-1}}{1+\tau_k}$, we have $\beta_k = \beta_0\prod_{i=1}^k\frac{1}{1+\tau_i} \leq \frac{2\beta_0}{k+2}$.
\Eproof
\end{proof}

\begin{lemma}\label{le:tech_lemma2}$($\cite[Lemma 4]{ZhuLiuTran2020} and \cite{TranDinh2015b}$)$
The following statements hold.
\begin{itemize}
\item[$\mathrm{(a)}$] For any $u, v, w\in\R^p$ and $t_1, t_2 \in \R$ such that $t_1 + t_2 \neq 0$, we have
\begin{equation*}
t_1\norms{u - w}^2 + t_2\norms{v - w}^2 = (t_1 + t_2)\norms{w - \tfrac{1}{t_1+t_2}(t_1u + t_2v)}^2 + \tfrac{t_1t_2}{t_1+t_2}\norms{u-v}^2.
\end{equation*}

\item[$\mathrm{(b)}$] For any $\tau \in (0, 1)$, $\hat{\beta}, \beta > 0$, $w, z\in\R^p$, we have
\begin{equation*}
\begin{array}{lcl}
\beta(1-\tau) \norms{w - z}^2 + \beta\tau\norms{w}^2 - (1-\tau)(\hat{\beta} - \beta)\norms{z}^2 & = & \beta\norms{w - (1-\tau)z}^2 \vspace{1ex}\\
&& + {~} (1-\tau)\big[\tau\beta - (\hat{\beta} - \beta) \big]\norms{z}^2.
\end{array}
\end{equation*}
\end{itemize}
\end{lemma}

The following lemma is a key step to address the strongly convex  case of $f$ in \eqref{eq:saddle_point_prob}.

\begin{lemma}\label{le:element_est12}
Given $L_k > 0$,  $\mu_f > 0$, and $\tau_k \in (0, 1)$, let $m_k := \frac{L_k + \mu_f}{L_{k-1} + \mu_f}$ and $a_k := \frac{L_k}{L_{k-1} + \mu_f}$. 
Assume that the following two conditions hold:
\begin{equation}\label{eq:element_cond2}
\arraycolsep=0.2em
\left\{\begin{array}{llcl}
&(1-\tau_k) \big[ \tau_{k-1}^2 + m_k\tau_k \big]&\geq & a_k\tau_k \vspace{1ex}\\
&m_k\tau_k\tau_{k-1}^2 + m_k^2\tau_k^2 &\geq & a_k\tau_{k-1}^2.
\end{array}\right.
\end{equation}
Let $\set{x^k}$ be a given sequence in $\R^p$.
We define $\hat{x}^k := x^k + \frac{1}{\omega_k}(x^k - x^{k-1})$, where $\omega_k$ is chosen such that
\begin{equation}\label{eq:x_extra_step}
\max\set{\frac{\tau_{k-1} + \sqrt{\tau_{k-1}^2 +  4a_k}}{2(1-\tau_{k-1})}, \frac{a_k\tau_k}{(1-\tau_k)(1-\tau_{k-1})\tau_{k-1}}} \leq \omega_k \leq \frac{\tau_{k-1}^2 + m_k\tau_k}{\tau_{k-1}(1-\tau_{k-1})}.
\end{equation}
Then, $\omega_k$ is well-defined, and for any $x \in \R^p$, we have
\begin{equation}\label{eq:element_est12}
\hspace{-2ex}
\arraycolsep=0.2em
\begin{array}{ll}
&L_k\tau_k^2\norms{\frac{1}{\tau_k}[\hat{x}^k - (1-\tau_k)x^k] - x}^2 -  \mu_f \tau_k(1-\tau_k)\norms{x^k - x}^2 \vspace{1ex}\\
&\qquad \leq {~} (1-\tau_k) \left(L_{k-1} + \mu_f\right)\tau_{k-1}^2\norms{\frac{1}{\tau_{k-1}}[x^k - (1-\tau_{k-1})x^{k-1}] - x}^2.
\end{array}
\hspace{-2ex}
\end{equation}
\end{lemma}

\begin{proof}
Firstly, from the definition $\hat{x}^k := x^k + \frac{1}{\omega_k}(x^k - x^{k-1})$ of $\hat{x}^k$, we have $\omega_k(\hat{x}^k - x^k) = x^k - x^{k-1}$.
Hence, we can show that
\begin{equation*}
\arraycolsep=0.2em
\begin{array}{lcl}
\tau_{k-1}^2\norms{\frac{1}{\tau_{k-1}}[x^k - (1-\tau_{k-1})x^{k-1}] - x}^2 &= & \norms{(1-\tau_{k-1})(x^k - x^{k-1}) + \tau_{k-1}(x^k - x)}^2 \vspace{1ex}\\
&= & \norms{(1-\tau_{k-1})\omega_k(\hat{x}^k - x^{k}) + \tau_{k-1}(x^k - x)}^2 \vspace{1ex}\\
&= & \omega_k^2(1-\tau_{k-1})^2\norms{\hat{x}^k - x^k}^2 +  \tau_{k-1}^2\norms{x^k - x}^2\vspace{1ex}\\
&& + {~} 2\omega_k(1-\tau_{k-1})\tau_{k-1}\iprods{\hat{x}^k - x^k, x^k - x}.
\end{array}
\end{equation*}
Alternatively, we also have
\begin{equation*}
\arraycolsep=0.2em
\begin{array}{lcl}
\tau_k^2\norms{\frac{1}{\tau_k}[\hat{x}^k - (1-\tau_k)x^k] - x}^2 & = &  \norms{\hat{x}^k - x^k}^2 + \tau_k^2\norms{x^k - x}^2  + 2\tau_k\iprods{\hat{x}^k - x^k, x^k - x}.
\end{array}
\end{equation*}
Utilizing the two last expressions, \eqref{eq:element_est12} can be rewritten equivalently to
\begin{equation*}
\arraycolsep=0.2em
\begin{array}{lcl}
\Tc_{[1]} &:= & 2\left[  \left(L_{k-1} + \mu_f\right)(1-\tau_k)(1-\tau_{k-1})\tau_{k-1}\omega_k - L_k\tau_k \right] \iprods{\hat{x}^k - x^k, x - x^k} \vspace{1ex}\\
&\leq & \left[ \left(L_{k-1} + \mu_f \right)(1-\tau_k)(1-\tau_{k-1})^2\omega_k^2 -  L_{k}\right] \norms{\hat{x}^k - x^k}^2 \vspace{1ex}\\
&& + {~} \left[ \left(L_{k-1} + \mu_f\right)(1-\tau_k)\tau_{k-1}^2 - L_k\tau_k^2 + \mu_f\tau_k(1-\tau_k)\right] \norms{x^k - x}^2.
\end{array}
\end{equation*} 
Now, let us denote
\begin{equation*}
\arraycolsep=0.2em
\left\{\begin{array}{lcl}
c_1 &:= & \left(L_{k-1}  + \mu_f\right)(1-\tau_k)(1-\tau_{k-1})\tau_{k-1}\omega_k - L_k\tau_k\vspace{1ex}\\
c_2 &:= & \left(L_{k-1} + \mu_f\right)(1-\tau_k)(1-\tau_{k-1})^2\omega_k^2 - L_k \vspace{1ex}\\
c_3 &:= & \left(L_{k-1} + \mu_f\right)(1-\tau_k)\tau_{k-1}^2 - L_k\tau_k^2 + \mu_f(1-\tau_k)\tau_k.
\end{array}\right.
\end{equation*}
Then, \eqref{eq:element_est12} is equivalent to 
\begin{equation}\label{eq:need_to_prove}
2c_1\iprods{\hat{x}^k - x^k, x - x^k} \leq c_2\norms{\hat{x}^k - x^k}^2 + c_3\norms{x - x^k}^2.
\end{equation}
Secondly, we need to guarantee that $c_1 \geq 0$.
This condition holds if we choose $\omega_k$ such that
\begin{equation}\label{eq:upper_bound_omegak2}
\omega_k \geq \frac{a_k\tau_k}{(1-\tau_k)(1-\tau_{k-1})\tau_{k-1}}.
\end{equation}
Thirdly, we also need to guarantee $c_2 \geq c_1$, which is equivalent to
\begin{equation*}
c_2 - c_1 =   \left(L_{k-1}  + \mu_f\right)(1-\tau_k)(1-\tau_{k-1})\left[(1-\tau_{k-1})\omega_k^2 - \tau_{k-1}\omega_k \right] - L_k(1-\tau_k)  \geq 0.
\end{equation*}
This condition holds if 
\begin{equation}\label{eq:upper_bound_omegak1}
\omega_k \geq \frac{\tau_{k-1} + \sqrt{\tau_{k-1}^2 +  4a_k}}{2(1-\tau_{k-1})}.
\end{equation}
Alternatively, we also need to guarantee $c_3 \geq c_1$, which is equivalent to
\begin{equation*}
c_3 - c_1 =   \left(L_{k-1}  + \mu_f\right)(1-\tau_k)\left[\tau_{k-1}^2 - (1-\tau_{k-1})\tau_{k-1}\omega_k \right] + (L_k + \mu_f)\tau_k(1-\tau_k) \geq 0.
\end{equation*}
This condition holds if
\begin{equation}\label{eq:lower_bound_omegak}
\omega_k \leq \frac{\tau_{k-1}^2 + m_k\tau_k}{\tau_{k-1}(1-\tau_{k-1})}.
\end{equation}
Combining \eqref{eq:upper_bound_omegak2}, \eqref{eq:upper_bound_omegak1}, and \eqref{eq:lower_bound_omegak}, we obtain
\begin{equation*}
\max\set{\frac{\tau_{k-1} + \sqrt{\tau_{k-1}^2 +  4a_k}}{2(1-\tau_{k-1})}, \frac{a_k\tau_k}{(1-\tau_k)(1-\tau_{k-1})\tau_{k-1}}} \leq \omega_k \leq \frac{\tau_{k-1}^2 + m_k\tau_k}{\tau_{k-1}(1-\tau_{k-1})},
\end{equation*}
which is exactly \eqref{eq:x_extra_step}.
Here, under the condition \eqref{eq:element_cond2}, the left-hand side of the last expression is less than or equal to the right-hand side.
Therefore, $\omega_k$ is well-defined.

Finally, under the choice of $\omega_k$ as in \eqref{eq:x_extra_step}, we have $c_2 \geq c_1 \geq 0$ and $c_3\geq c_1 \geq 0$.
Hence, \eqref{eq:need_to_prove} holds, which is also equivalent to \eqref{eq:element_est12}.
\Eproof
\end{proof}

\beforesec
\section{Appendix 2: Technical Proof of Lemmas \ref{le:descent_pro} and \ref{le:maintain_gap_reduction1}  in Section~\ref{sec:alg_scheme}}\label{apdx:sec:A2}
\aftersec
This section provides the full proof of Lemma~\ref{le:descent_pro} and Lemma~\ref{le:maintain_gap_reduction1} in the main text.

\beforesubsec
\subsection{The proof of Lemma~\ref{le:descent_pro}: Key estimate of the primal-dual step~\eqref{eq:prox_grad_step0}}\label{apdx:le:descent_pro}
\aftersubsec
\begin{proof}
From the first line of \eqref{eq:prox_grad_step0} and Lemma~\ref{le:g_properties}(a),  we have $\nabla_ug_{\beta_k}(K\hat{x}^k, \dot{y}) = K^{\top}y^{k+1}$.
Now, from the second line of \eqref{eq:prox_grad_step0}, we also have 
\begin{equation*}
0 \in \partial{f}(x^{k+1}) + L_k(x^{k+1} - \hat{x}^k) + K^{\top}\nabla_u{g_{\beta_k}}(K\hat{x}^k, \dot{y}).
\end{equation*}
Combining this inclusion and the $\mu_f$-convexity of $f$, for any $x\in\dom{f}$, we get
\begin{equation*}
\begin{array}{lcl}
f(x^{k+1}) & \leq & f(x) + \iprods{\nabla_u{g_{\beta_k}}(K\hat{x}^k, \dot{y}), K(x - x^{k+1})}  + L_k\iprods{x^{k+1} - \hat{x}^k, x - x^{k+1}} \vspace{1ex}\\
&& - {~} \frac{\mu_f}{2}\norms{x^{k+1} - x}^2.
\end{array}
\end{equation*}
Since $g_{\beta}(\cdot, \dot{y})$ is $\frac{1}{\beta + \mu_{g^{*}}}$-smooth by Lemma~\ref{le:g_properties}(a), for any $x\in\dom{f}$, we have
\begin{equation*}
\arraycolsep=0.2em
\begin{array}{lcl}
g_{\beta_k}(Kx^{k+1}, \dot{y}) & \leq & g_{\beta_k}(K\hat{x}^k, \dot{y}) + \iprods{\nabla_u{g}_{\beta_k}(K\hat{x}^k, \dot{y}), K(x^{k+1} - \hat{x}^k)} \vspace{1ex}\\
&& + {~} \frac{1}{2(\beta_k + \mu_{g^{*}})}\Vert K(x^{k+1} - \hat{x}^k)\Vert^2 \vspace{1ex}\\
& = & g_{\beta_k}(K\hat{x}^k, \dot{y}) + \iprods{\nabla_ug_{\beta_k}(K\hat{x}^k, \dot{y}), K(x - \hat{x}^k)}  \vspace{1ex}\\
&& - {~}  \iprods{\nabla_u{g}_{\beta_k}(K\hat{x}^k, \dot{y}), K(x - x^{k+1})} \vspace{1ex}\\
&& + {~} \frac{1}{2(\mu_{g^{*}} + \beta_k)}\Vert K(x^{k+1} - \hat{x}^k)\Vert^2.
\end{array}
\end{equation*}
Now, combining the last two estimates, we get
\begin{equation}\label{eq:proof_est1}
\hspace{-1ex}
\arraycolsep=0.2em
\begin{array}{lcl}
f(x^{k+1}) + g_{\beta_k}(Kx^{k+1}, \dot{y}) & \leq & f(x) + g_{\beta_k}(K\hat{x}^k, \dot{y}) +  \iprods{\nabla_ug_{\beta_k}(K\hat{x}^k, \dot{y}), K(x - \hat{x}^k)} \vspace{1ex}\\
&& + {~} L_k\iprods{x^{k+1} - \hat{x}^k, x - \hat{x}^k}   -  L_k\norms{x^{k+1} - \hat{x}^k}^2 \vspace{1ex}\\
&& + {~}  \frac{1}{2(\mu_{g^{*}} + \beta_k)}\Vert K(x^{k+1} - \hat{x}^k)\Vert^2   - \frac{\mu_f}{2}\norms{x - x^{k+1}}^2.
\end{array}
\hspace{-5ex}
\end{equation}
Using Lemma~\ref{le:g_properties}(a) again, we have
\begin{equation}\label{eq:proof_est2}
\arraycolsep=0.2em
\hspace{-2ex}
\begin{array}{lcl}
\ell_{\beta_k}(x^k, \dot{y}) &:= &  g_{\beta_k}(K\hat{x}^k, \dot{y}) + \iprods{\nabla_ug_{\beta_k}(K\hat{x}^k, \dot{y}), K(x^k - \hat{x}^k)} \vspace{1ex}\\
&\leq & g_{\beta_k}(Kx^k, \dot{y})   -  \frac{\beta_k+\mu_{g^{*}}}{2}\norms{\nabla_ug_{\beta_k}(K\hat{x}^k, \dot{y}) - \nabla_ug_{\beta_k}(Kx^k, \dot{y})}^2.
\end{array}
\hspace{-2ex}
\end{equation}
Substituting $x := x^k$ into \eqref{eq:proof_est1}, and multiplying the result by $1-\tau_k$ and adding the result to  \eqref{eq:proof_est1} after multiplying it by $\tau_k$, then using 
\eqref{eq:proof_est2}, we can derive
\begin{equation}\label{eq:proof_est3} 
\hspace{-1ex}
\arraycolsep=0.2em
\begin{array}{lcl}
F_{\beta_k}(x^{k+1}, \dot{y}) &:= &  f(x^{k+1}) + g_{\beta_k}(Kx^{k+1}, \dot{y}) \vspace{1ex}\\
& \leq & (1 - \tau_k)[ f(x^k) + g_{\beta_k}(Kx^k, \dot{y})]    +  \tau_k\left[f(x) + \ell_{\beta_k}(x, \dot{y}) \right] \vspace{1ex}\\
&& - {~} L_k\norms{x^{k+1} - \hat{x}^k}^2 + \frac{1}{2(\mu_{g^{*}}+\beta_k)}\Vert K(x^{k+1} - \hat{x}^k)\Vert^2  \vspace{1ex}\\
&& + {~} L_k\iprods{x^{k+1} - \hat{x}^k, \tau_kx - \hat{x}^k +  (1-\tau_k)x^k} \vspace{1ex}\\
&& - {~} \frac{\mu_f}{2}\left[(1-\tau_k)\norms{x^{k+1} - x^k}^2 + \tau_k\norms{x - x^{k+1}}^2\right] \vspace{1ex}\\
&& - {~} \frac{(1-\tau_k)(\beta_k+\mu_{g^{*}})}{2}\norms{\nabla_ug_{\beta_k}(K\hat{x}^k, \dot{y}) - \nabla_ug_{\beta_k}(Kx^k, \dot{y})}^2.
\end{array}
\hspace{-5ex}
\end{equation}
From Lemma~\ref{le:tech_lemma2}(a), we can easily show that
\begin{equation*}
\begin{array}{lcl}
(1-\tau_k)\norms{x^{k+1} - x^k}^2 + \tau_k\norms{x^{k+1} - x}^2 & = &   \tau_k^2\norms{\tfrac{1}{\tau_k}[ x^{k+1} - (1-\tau_k)x^k] - x}^2 \vspace{1ex}\\
&& + {~} \tau_k(1-\tau_k)\norms{x - x^k}^2.
\end{array}
\end{equation*}
We also have the following elementary relation
\begin{equation*}
\arraycolsep=0.1em
\begin{array}{lcl}
\iprods{x^{k+1} - \hat{x}^k, \tau_k x - [\hat{x}^k - (1-\tau_k)x^k]}  &= & \frac{\tau_k^2}{2}\norms{\tfrac{1}{\tau_k}[\hat{x}^k - (1-\tau_k)x^k] - x}^2   +  \frac{1}{2}\norms{x^{k+1} - \hat{x}^k}^2 \vspace{1ex}\\
&&  - {~} \frac{\tau_k^2}{2}\norms{\tfrac{1}{\tau_k}[x^{k+1} - (1-\tau_k)x^k] - x}^2.
\end{array}
\end{equation*}
Substituting the two last expressions into \eqref{eq:proof_est3}, we obtain
\begin{equation}\label{eq:proof4} 
\arraycolsep=0.2em
\begin{array}{lcl}
F_{\beta_k}(x^{k+1}, \dot{y})  & \leq & (1 - \tau_k) F_{\beta_k}(x^k,  \dot{y})  + \tau_k \left[ f(x)  + \ell_{\beta_k}(x, \dot{y}) \right] \vspace{1ex}\\
&& + {~} \frac{L_k\tau_k^2}{2} \norms{\tfrac{1}{\tau_k}[\hat{x}^k - (1-\tau_k)x^k] - x}^2 \vspace{1ex}\\
&& - {~} \frac{\tau_k^2}{2}\left( L_k + \mu_f \right) \norms{\tfrac{1}{\tau_k}[x^{k+1} - (1-\tau_k)x^k] - x}^2  \vspace{1ex}\\
&& - {~} \frac{(1-\tau_k)(\mu_{g^{*}} + \beta_k)}{2}\norms{\nabla_ug_{\beta_k}(K\hat{x}^k, \dot{y}) - \nabla_ug_{\beta_k}(Kx^k, \dot{y})}^2 \vspace{1ex}\\
&& - {~} \frac{L_k}{2}\norms{x^{k+1} - \hat{x}^k}^2 + \frac{1}{2(\mu_{g^{*}}+\beta_k)}\Vert K(x^{k+1} - \hat{x}^k)\Vert^2  \vspace{1ex}\\
&&  - {~} \frac{\mu_f(1-\tau_k)\tau_k}{2}\norms{x - x^k}^2.
\end{array}
\hspace{-2ex}
\end{equation}
One the one hand, by \eqref{eq:key_bound_apdx4} of Lemma~\ref{le:g_properties}, we have 
\begin{equation*}
F_{\beta_k}(x^k, \dot{y}) \leq F_{\beta_{k-1}}(x^k, \dot{y}) + \frac{(\beta_{k-1} - \beta_k)}{2}\norms{\nabla_ug_{\beta_{k}}(Kx^k, \dot{y}) - \dot{y}}^2.
\end{equation*}
On the other hand, by \eqref{eq:key_bound_apdx5} of Lemma~\ref{le:g_properties}, we get
\begin{equation*}
f(x) + \ell_{\beta_k}(x, \dot{y}) \leq \Lc(x, y^{k+1}) - \frac{\beta_k}{2}\norms{ \nabla_ug_{\beta_{k}}(K\hat{x}^k, \dot{y}) - \dot{y}}^2,
\end{equation*}
where $\Lc(x, y^{k+1}) := f(x) +  \iprods{Kx, y^{k+1}} - g^{*}(y^{k+1})$ is the Lagrange function in \eqref{eq:saddle_point_prob}.

Now, substituting the last two inequalities into \eqref{eq:proof4}, and using Lemma~\ref{le:tech_lemma2}(b) with $w := \nabla_ug_{\beta_k}(K\hat{x}^k, \dot{y}) - \dot{y}$ and $z := \nabla_ug_{\beta_k}(Kx^k, \dot{y}) - \dot{y}$, we arrive at
\begin{equation*} 
\arraycolsep=0.2em
\begin{array}{lcl}
F_{\beta_k}(x^{k+1}, \dot{y})  & \leq & (1 - \tau_k) F_{\beta_{k-1}}(x^k, \dot{y})  + \tau_k\Lc(x, y^{k+1}) + \frac{L_k\tau_k^2}{2} \norms{\tfrac{1}{\tau_k}[\hat{x}^k - (1-\tau_k)x^k] - x}^2 \vspace{1ex}\\
&& - {~} \frac{\tau_k^2}{2}\left( L_k + \mu_f \right) \norms{\tfrac{1}{\tau_k}[x^{k+1} - (1-\tau_k)x^k] - x}^2 - \frac{\mu_f(1-\tau_k)\tau_k}{2}\norms{x - x^k}^2 \vspace{1ex}\\
&& - {~} \frac{L_k}{2}\norms{x^{k+1} - \hat{x}^k}^2 + \frac{1}{2(\mu_{g^{*}}+\beta_k)}\Vert K(x^{k+1} - \hat{x}^k)\Vert^2  \vspace{1ex}\\
&&  - {~} \frac{(1-\tau_k)}{2}\left[ \tau_k\beta_k - (\beta_{k-1} - \beta_k)\right]\norms{ \nabla_ug_{\beta_k}(Kx^k, \dot{y})  - \dot{y}}^2 \vspace{1ex}\\
&& - {~} \frac{\beta_k}{2}\norms{\nabla_ug_{\beta_k}(K\hat{x}^k, \dot{y}) - \dot{y} - (1-\tau_k)\left[ \nabla_ug_{\beta_k}(Kx^k, \dot{y})  - \dot{y} \right]}^2 \vspace{1ex}\\
&& - {~} \frac{(1-\tau_k)\mu_{g^{*}}}{2}\norms{\nabla_ug_{\beta_k}(K\hat{x}^k, \dot{y}) - \nabla_ug_{\beta_k}(Kx^k, \dot{y})}^2.
\end{array}
\end{equation*}
By dropping the last two nonpositive terms in the last inequality, we obtain \eqref{eq:key_est00_a}. 
\Eproof
\end{proof}

\beforesubsec
\subsection{The proof of Lemma~\ref{le:maintain_gap_reduction1}: Recursive estimate of the Lyapunov function}\label{apdx:le:maintain_gap_reduction1}
\aftersubsec
\begin{proof}
First, from the last line $\tilde{y}^{k+1} = (1-\tau_k)\tilde{y}^k + \tau_ky^{k+1}$ of \eqref{eq:prox_grad_scheme}, and the $\mu_{g^{*}}$-convexity of $g^{*}$, we have 
\begin{equation*}
\begin{array}{lcl}
\Lc(x, \tilde{y}^{k+1}) & := & f(x) + \iprods{Kx, \tilde{y}^{k+1}} - g^{*}(\tilde{y}^{k+1}) \vspace{1ex}\\
& \geq & (1-\tau_k)\Lc(x, \tilde{y}^k) + \tau_k\Lc(x, y^{k+1}) + \frac{\mu_{g^{*}}\tau_k(1-\tau_k)}{2}\norms{y^{k+1} - \tilde{y}^k}^2.
\end{array}
\end{equation*}
Hence, $\tau_k\Lc(x, y^{k+1}) \leq \Lc(x, \tilde{y}^{k+1}) - (1-\tau_k)\Lc(x, \tilde{y}^k) - \frac{\mu_{g^{*}}\tau_k(1-\tau_k)}{2}\norms{y^{k+1} - \tilde{y}^k}^2$.
Substituting this estimate into \eqref{eq:key_est00_a} and dropping the term $- \frac{\mu_{g^{*}}\tau_k(1-\tau_k)}{2}\norms{y^{k+1} - \tilde{y}^k}^2$, we can derive
\begin{equation}\label{eq:key_est00_proof1}
\hspace{-2ex}
\arraycolsep=0.1em
\begin{array}{lcl}
F_{\beta_k}(x^{k+1},\dot{y})  & \leq & (1 - \tau_k) F_{\beta_{k-1}}(x^k, \dot{y})  + \Lc(x, \tilde{y}^{k+1}) - (1-\tau_k)\Lc(x, \tilde{y}^k) \vspace{1ex}\\
&& + {~} \frac{L_k\tau_k^2}{2} \big\Vert \tfrac{1}{\tau_k}[\hat{x}^k - (1-\tau_k)x^k] - x \big\Vert^2 \vspace{1ex}\\
&& - {~} \frac{\tau_k^2}{2}\left(L_k + \mu_f\right) \big\Vert \tfrac{1}{\tau_k}[x^{k+1} - (1-\tau_k)x^k] - x \big\Vert^2  \vspace{1ex}\\
&& - {~} \frac{L_k}{2}\norms{x^{k+1} - \hat{x}^k}^2 + \frac{1}{2(\mu_{g^{*}} + \beta_k)}\Vert K(x^{k+1} - \hat{x}^k)\Vert^2 \vspace{1ex}\\
&& - {~} \frac{\mu_f\tau_k(1-\tau_k)}{2}\norms{x^k - x}^2.
\end{array}
\hspace{-2ex}
\end{equation}
Now, it is obvious to show that the condition \eqref{eq:param_cond} is equivalent to the condition \eqref{eq:element_cond2} of Lemma~\ref{le:element_est12}. 
In addition, we choose $\eta_k = \frac{1}{\omega_k}$ in our update \eqref{eq:Lk_m_k}, where $\omega_k := \frac{\tau_{k-1}^2 + m_k\tau_k}{\tau_{k-1}(1-\tau_{k-1})}$, which is the upper bound of \eqref{eq:x_extra_step}.
Hence, \eqref{eq:x_extra_step} automatically holds.
Using \eqref{eq:element_est12}, we have
\begin{equation*}
\begin{array}{lcl}
\Tc_{[2]} &:= & \frac{L_k\tau_k^2}{2} \big\Vert \tfrac{1}{\tau_k}[\hat{x}^k - (1-\tau_k)x^k] - x \big\Vert^2 - \frac{\mu_f\tau_k(1-\tau_k)}{2}\norms{x^k - x}^2 \vspace{1ex}\\
&\leq & \frac{\tau_{k-1}^2}{2} (1-\tau_k)\left(L_{k-1} + \mu_f\right)\big\Vert \tfrac{1}{\tau_{k-1}}[x^{k} - (1-\tau_{k-1})x^{k-1}] - x \big\Vert^2.
\end{array}
\end{equation*}
Moreover, $\frac{1}{2(\mu_{g^{*}} + \beta_k)}\Vert K(x^{k+1} - \hat{x}^k)\Vert^2 \leq \frac{\norms{K}^2}{2(\mu_{g^{*}} + \beta_k)}\Vert x^{k+1} - \hat{x}^k\Vert^2 = \frac{L_k}{2}\norms{x^{k+1} - \hat{x}^k}^2$ due to the definition of $L_k$ in \eqref{eq:Lk_m_k}.
Substituting these two estimates into \eqref{eq:key_est00_proof1}, and utilizing the definition \eqref{eq:lyapunov_func1} of $\Vc_k$,  we obtain \eqref{eq:key_est1_ncvx}.
\Eproof
\end{proof}

\bibliographystyle{plain}

\end{document}